\documentclass[english]{article}
\usepackage{layout}
\usepackage{babel}
\usepackage[OT1]{fontenc}
\usepackage{graphicx}
\usepackage{endnotes}
\usepackage{latexsym}
\usepackage{amsmath}
\usepackage{amssymb}
\usepackage{amsthm}
\usepackage[mathscr]{eucal}
\usepackage{array}
\usepackage{amsfonts}%
\usepackage{centernot}
\usepackage{multirow}
\usepackage{enumerate}
\usepackage{amsbsy}
\usepackage{epsfig}

\newtheorem{defn}{Definition}
\newtheorem{theorem}{Theorem}
\newtheorem{prop}{Proposition}
\newtheorem{lemma}{Lemma}
\newtheorem{corollary}{Corollary}

\newtheorem{preremark}{Remark}

\newcount{\myi}
\newcommand{\Repeat}[2]{%
    \myi=0
    \loop
        \ifnum\myi<#2
        #1
        \advance\myi by 1
    \repeat
}

\def\Bbb{\bf}

\def\RR{{\Bbb R}}

\def\W2{W^{1,2}({\cal O}(M))}

\def\1half{\frac{1}{2}}

\def\Dom{\mathop{\mbox{\normalfont Dom}}\nolimits}
\newcommand{\beq}{\begin{equation*}}
\newcommand{\eeq}{\end{equation*}}

\newcommand{\lp}{\left(}
\newcommand{\rp}{\right)}

\newcommand{\Levy}{L\'evy}

\newcommand{\Ssmooth}{\boldsymbol{\mathcal{S}}}

\newcommand{\gF}{g_F}

\newcommand{\Lspa}{\mathbb{L}}

\newcommand{\arrowlim}[2]{\xrightarrow[#1\rightarrow #2]{}}

\newcommand{\hilbert}{\mathfrak{H}}

\newcommand{\Hclass}{\mathscr{H}}
\newcommand{\Wclass}{\mathscr{W}}
\newcommand{\FMclass}{\mathscr{FM}}
\newcommand{\PZclass}{\mathscr{PZ}}

\newcommand{\NN}{\mathbb{N}}

\newcommand{\Hess}{\text{\upshape{Hess}}}

\newcommand{\mean}[1]{\mathbb{E}\left[ #1 \right]}

\newcommand{\Var}[1]{\operatorname*{Var}\left[ #1 \right]}
\newcommand{\Cov}[2]{\operatorname*{Cov}\left[ #1,#2 \right]}
\newcommand{\meanW}[1]{\mathbb{E}^W\left[ #1 \right]}
\newcommand{\meanJ}[1]{\mathbb{E}^J\left[ #1 \right]}

\newcommand{\MDer}{\boldsymbol{D}}

\newcommand{\domDWJ}{$\Dom\boldsymbol{D}^{W,J}$}

\newcommand{\Rinnprod}[3]{\left\langle #1,#2 \right\rangle_{\RR^{#3}}}
\newcommand{\HSnorm}[1]{\left\| #1\right\|_{H.S.}}
\newcommand{\HSinnprod}[2]{\left\langle #1,#2 \right\rangle_{H.S.}}

\newcommand{\Lipnorm}[1]{\left\| #1\right\|_{\text{\upshape{Lip}}}}
\newcommand{\BLnorm}[1]{\left\| #1\right\|_{\text{\upshape{BL}}}}
\newcommand{\normsup}[1]{\left\| #1\right\|_\infty}

\newcommand{\Lqnormmu}[2]{\left\| #1 \right\|_{\mathbb{L}^{2}_{#2}}}

\newcommand{\innprod}[2]{\left\langle #1,#2 \right\rangle_{\hilbert}}

\newcommand{\Levyp}{\mathcal{L}}
\newcommand{\abs}[1]{\left| #1 \right|}
\newcommand{\Hprod}[2]{\left\langle #1 , #2 \right\rangle_{\hilbert}}
\newcommand{\HHprod}[2]{\left\langle #1 , #2 \right\rangle_{\hilbert^{\otimes 2}}}
\newcommand{\Hnorm}[1]{\left\| #1 \right\|_{\hilbert}}
\newcommand{\HHnorm}[1]{\left\| #1 \right\|_{\hilbert^{\otimes 2}}}

\newcommand{\Lprod}[3]{\left\langle #1 , #2 \right\rangle_{\Lspa^2_{#3}}}

\newcommand{\Lqnorm}[2]{\left\| #1 \right\|_{\Lspa^{2}_{#2}}}
\newcommand{\Lqprod}[3]{\left\langle #1 , #2 \right\rangle_{\Lspa^{2}_{#3}}}
\newcommand{\opnorm}[1]{\left\|#1\right\|_{\text{op}}}

\newcommand{\dint}{\int\hspace{-.2cm}\int}
\newcommand{\be}{\begin{equation}}
\newcommand{\ee}{\end{equation}}

\begin{document}

\title{Normal Convergence Using Malliavin Calculus\\With Applications and Examples}
%
%
%
\author{Juan Jos\'e V\'iquez R}

%
%
%
%
%
\maketitle

\begin{abstract}
We prove the chain rule in the more general framework of the Wiener-Poisson space, allowing us to obtain the so-called Nourdin-Peccati bound. From this bound we obtain a second-order Poincar\'e-type inequality that is useful in terms of computations. For completeness we survey these results on the Wiener space, the Poisson space, and the Wiener-Poisson space. We also give several applications to central limit theorems with relevant examples: linear functionals of Gaussian subordinated fields (where the subordinated field can be processes like fractional Brownian motion or the solution of the Ornstein-Uhlenbeck SDE driven by fractional Brownian motion), Poisson functionals in the first Poisson chaos restricted to infinitely many ``small'' jumps (particularly fractional L\'evy processes) and the product of two Ornstein-Uhlenbeck processes (one in the Wiener space and the other in the Poisson space). We also obtain bounds for their rate of convergence to normality.
\end{abstract}

\section{Introduction}

In recent years many papers have looked at combining Stein's method with Malliavin calculus in order to uncover new tools for proving various central limit theorems (CLTs). For example, I. Nourdin and G. Peccati derived an upper bound (NP bound) for the Wasserstein (Kantorovich) distance (and other distances) on the Wiener space using Stein's equation \cite{Nourdin0}. Later, the same authors along with G. Reinert derived a second order Poincar\'e(-type) inequality which is useful (in terms of computations) for proving CLTs, and which in fact can be seen as a quantitative extension of Stein's method from which upper bounds for the rate of convergence to normality can be found \cite{Nourdin}. The first two authors with A. R\'eveillac extended these results to the multidimensional case \cite{Nourdinmult}. In \cite{Peccati}, G. Peccati, J. L. Sol\'e, M. S. Taqqu and F. Utzet, were able to find an upper bound, similar to the one in \cite{Nourdin0}, for the Wasserstein distance in the Poisson space. G. Peccati and C. Zheng succeeded in extending this to the multi-dimensional case in \cite{Peccati2}. All these works are important as they give quantitative tools for computing whether a random variable converges to normality, and if so, its rate of convergence.

\bigskip

A question naturally arises --- can this be done for a general L\'evy process; that is, is this upper bound achievable in a mixed space: the Wiener-Poisson space? The main difficulty in answering this question is that in the Wiener-Poisson space we do not yet have a global chain rule. Neither we have a decomposition in orthogonal polynomials (unlike with Hermite polynomials in the Wiener space, see \cite{NunnoOrthpol} for a complete explanation), nor results like the equivalence between the Mehler semigroup and the Ornstein-Uhlenbeck semigroup. So, to overcome these shortcomings, we must deduce new formulas that will allow us to follow the ideas developed in \cite{Nourdin0} and \cite{Nourdin}, and recover their results for the Wiener-Poisson space. We will show that this bound still holds even when both spaces are involved. One of the major contributions of this article is the development of the unified chain rule on the Wiener-Poisson space which allows the reproduction of the Nourdin-Peccati theory on this more general space.

\bigskip

Before get into the details, some notation: Let $L_t$ be a L\'evy process ($L_t$ has stationary and independent increments, is continuous in probability and $L_0=0$, with $\mean{L_1^2}<\infty$) with L\'evy-triplet given by $(0,\sigma^2,\nu)$, where $\nu$ is the L\'evy measure. The measure $\mu$ on the underlying Hilbert space $\Lspa^2_{\mu}$ is defined by the underlying L\'evy process, that is, for any $z=(t,x)\in\RR^+\times\RR_0$ we have $d\mu(z)=\sigma^2\delta_0(x)\,dt+x^2\,\nu(dx)\,dt$, where $\RR_0=\RR-\{0\}$. Then 
\beq
\dint_{\RR^+\times\RR}f(z) \,d\mu(z)=\sigma^2\int_{\RR^+}f(t,0) \,dt+\dint_{\RR^+\times\RR_0}f(t,x) x^2\,dt\,d\nu(x).
\eeq
On the other hand, in order to define a Malliavin derivative in the Wiener-Poisson space it is sufficient to have a chaos decomposition of the space $\Lspa^2(\Omega)$. This was achieved in \cite{Ito} by K. It$\hat{\text{o}}$, so any random variable in $\Lspa^2(\Omega)$ has a projection on the $q^{\text{th}}$ chaos given by $I_q(f_q)$, where $f_q$ is a symmetric function in $\Lspa^2_{\mu^{\otimes q}}$. Also, when working in the Wiener-Poisson space, the Malliavin derivative can be regarded in terms of ``directions'', i.e., we can think of it as the derivative in the Wiener direction or the derivative in the Poisson direction. The fact that this can be done in this way is shown in \cite{Josep} by J. L. Sol\'e, F. Utzet and J. Vives (a quick review of the theory is given below). They explain that the Malliavin derivative with parameter $z\in\RR^+\times\RR$ can be split into two cases, when $z=(t,0)$ and when $z=(t,x)$ with $x\neq0$. The first case will be the derivative in the Wiener direction (intuitively because there are no jumps when $x=0$), and the second will be the derivative in the Poisson direction. A distinction between the Malliavin calculus in the Wiener space or the Poisson space and in this Wiener-Poisson space is the need to define two subspaces of $\Lspa^2(\Omega)$: one where the Malliavin derivative in the Wiener direction coincides with the usual Malliavin derivative in the Wiener space and is well defined, and another where the Malliavin derivative in the Poisson direction is well defined. This suitable subspace is denoted by $\Dom\boldsymbol{D}^{W,J}$.

\bigskip

\noindent {\bf Theorem 1.}{\em (Main result: NP Bound in Wiener-Poisson Space)\\
Suppose that $Z\sim\mathcal{N}(0,\Sigma)$ with a positive definite covariance matrix $\Sigma$. Let $F=(F^{(1)},\dots,F^{(d)})$ be such that $\mean{F}=0$ and $F^{(i)}\in$ \domDWJ, for all $i$. Then, for a distance $d_{\Hclass}$ with respect to a suitable separating class $\Hclass$,
\beq
d_{\Hclass}(Z,F)\leq \boldsymbol{k}\left(\mean{\HSnorm{\Sigma-\gF(F)}}+\mean{\Lprod{\abs{x}\left(\sum_{i=1}^d\abs{\MDer F^{(i)}}\right)^2}{\sum_{i=1}^d\abs{\MDer L^{-1}F^{(i)}}}{\mu}}\right),
\eeq
where $\gF^{i,j}(x):=\mean{\left.\innprod{\MDer F^{(i)}}{-\MDer L^{-1}F^{(j)}}\right|F=x}$, $\MDer$ is the Malliavin derivative, $L^{-1}$ is the pseudo-inverse of the infinitesimal generator of the Ornstein-Uhlenbeck semigroup, and $\hilbert$ is the underlying Hilbert space.}

\bigskip

In the case of the Wiener space upper bound, $\MDer$ is the Malliavin derivative defined in that space, so, this inequality holds even when the underlying Hilbert space is not $\Lspa^2_{\mu}$. Also, since there are no jumps here, the second term on the right disappears ($x$ is the size of the jump). In the Poisson space case, since we do not (yet) have a Malliavin calculus theory developed for a general abstract Hilbert space, the underlying Hilbert space must be $\hilbert = \Lspa^2_{\mu}$. Thanks to these (Wiener, Poisson and Wiener-Poisson) NP bounds, many CLTs can be proved and generalizations can be made. In this paper, these bounds are reviewed for each space, showing their importance by giving applications with relevant examples.

\bigskip

In the Wiener space case the second order Poincar\'e inequality is used to prove normal convergence for linear functionals of Gaussian-subordinated fields when the decay rate of the covariance function of the underlying Gaussian process satisfies certain conditions. These CLTs are applied to the important cases where the underlying Gaussian process is either the fractional Brownian motion or the fractional-driven Ornstein-Uhlenbeck process, with $H\in\bigl(0,\frac{1}{2}\bigr)\cup\bigl(\frac{1}{2},1\bigr)$.

\bigskip

In the Poisson space case, the respective upper bound is used to prove that the small jumps process (jumps with length less than or equal to $\epsilon$) of a Poisson functional process with infinitely many jumps goes to a normal random variable when $\epsilon$ goes to zero. Furthermore, we prove a remarkable extension of the known result (proved in \cite{Asmussen}) which states that the small jumps process of a L\'evy process can be approximated by Brownian motion as $\epsilon$ goes to zero. It is extended to Poisson functionals $\bigl(I_1(f)\bigr)$ and showed that the small jumps process of this functional can be approximated by a Gaussian functional with the same kernel $f$ as $\epsilon$ goes to zero. Then this result is applied to show that in order to simulate a fractional (pure jump) L\'evy process (fLp), it is sufficient to simulate a process with finitely many jumps plus an independent fractional Brownian motion (fBm).

\bigskip

Finally, the second order Poincar\'e(-type) inequality, developed in this paper, is used to prove that the time average of the product of a Wiener Ornstein-Uhlenbeck process with a Poisson Ornstein-Uhlenbeck process converges to a normal random variable as time goes to infinity. This example highlights the importance of the inequality in the Wiener-Poisson space, since it cannot be achieved by the NP bounds in the Wiener or Poisson spaces individually. An estimate of the rate of convergence to normality is obtained in the examples where the second order Poincar\'e inequalities are used.

\bigskip

The paper is organized as follows: In Section 2 we recall the basic tools of Malliavin calculus on the Wiener space and state the Malliavin calculus results for the Wiener-Poisson space. In Section 3 the general chain rules for the Wiener-Poisson space are proven. Finally, the theory developed in \cite{Nourdin} and \cite{Nourdin0} is extended using the Stein's method and the so-called Nourdin-Peccati analysis but for the Wiener-Poisson space, and within this framework we state a ``L\'evy version'' of the second order Poincar\'e inequality. Section 4 is dedicated to going over the inequalities for the Wiener, Poisson and Wiener-Poisson spaces. In the Wiener space case, is extended a result proved in \cite{Nourdin} concerning CLTs of linear functionals of Gaussian-subordinated fields. In the Poisson space case, is given a result on the simulation of small jumps for processes with infinitely many jumps. Finally, is shown an example of an application of the second order Poincar\'e inequality in the Wiener-Poisson space.

\section{Preliminaries}

As mentioned above, a useful tool for proving normal convergence on the Wiener space is the NP bound developed in \cite{Nourdin}. This requires various Malliavin calculus results on the Wiener space (Malliavin derivative, contraction of order $r$, Mehler formula, etc.) which are extensively studied and explained in \cite{Nualart}. For the sake of completeness, the basic tools from Malliavin calculus in the Wiener space are reviewed and then is introduced the Malliavin calculus	in the Wiener-Poisson setting, both needed in this article.

\subsection{Malliavin Calculus on Wiener space}

In the following we will introduce the theory of Malliavin calculus as presented in \cite{Nualart}. Let $(\Omega,\mathcal{F},\mathbb{P})$ be a probability space, where $W:=\{W(h) \mid h\in\hilbert\}$ is an isonormal Gaussian process with $\hilbert$ as a real separable Hilbert space, that is, $W$ is a centered Gaussian family such that $E[W(h_1)W(h_2)]=\Hprod{h_1}{h_2}$. Choose $\mathcal{F}$ to be the $\sigma$-algebra generated by $W$. Let $H_q$ be the $q^{\text{th}}$ Hermite polynomial, $H_q(x)=(-1)^qe^{\frac{x^2}{2}}\frac{\partial^q}{\partial x^q}\bigl(e^{-\frac{x^2}{2}}\bigr)$, and define the $q^{\text{th}}$ Wiener chaos of $W$ (denoted by $\mathbb{H}_q$) as the subspace of $\Lspa^2(\Omega):=\Lspa^2(\Omega,\mathcal{F},\mathbb{P})$ generated by $\{H_q(W(h))\mid h\in\hilbert,\Hnorm{h}=1\}$. It is important to emphasize that $\Lspa^2(\Omega)$ can be decomposed (Wiener chaos expansion) into an infinite orthogonal sum of the spaces $\mathbb{H}_q$: $\Lspa^2(\Omega)=\oplus_{q=0}^\infty \mathbb{H}_q$.

\begin{preremark}\label{chaosintegral}
In the case where $\hilbert=\Lspa^2_\mu$, for any $F\in \Lspa^2(\Omega)$,
\begin{align}\label{chaosrep0}
F=\sum_{q=0}^\infty I_q(f_q),
\end{align}
where $I_q$ is the $q^{\text{th}}$ multiple stochastic integral, $f_0=\mean{F}$, $I_0$ is the identity mapping on constants and $f_q\in \Lspa^2_{\mu^{\otimes q}}$ are symmetric functions uniquely determined by $F$.
\end{preremark}
Let $\mathcal{S}$ be the class of smooth random variables, i.e., if $F\in\mathcal{S}$ then there exists a function $\phi\in C^\infty(\RR^n)$ such that $\frac{\partial^k \phi}{\partial x_i^k}(x)$ has polynomial growth for all $k=0,1,2,3,\dots$ and $F=\phi(W(h_1),\dots,W(h_n))$, $h_i\in\hilbert$. The Malliavin derivative of $F\in \mathcal{S}$ with respect to $W$ is the element of $\Lspa^2(\Omega,\hilbert)$ defined as
\begin{align}\label{malliavinderivative}
\boldsymbol{D}F=\sum_{i=1}^n\frac{\partial \phi}{\partial x_i}\bigl(W(h_1),\dots,W(h_n)\bigr)h_i.
\end{align}
In particular $\boldsymbol{D}W(h)=h$ for every $h\in\hilbert$. Notice that in this particular case we have an explicit relation between the covariance of $W$ and the inner product of the Malliavin derivate, $\text{Cov}[W(h_1)W(h_2)]=E[W(h_1)W(h_2)]=\Hprod{h_1}{h_2}=\Hprod{\boldsymbol{D}W(h_1)}{\boldsymbol{D}W(h_2)}$.
\begin{preremark}\label{hilbertderivative}
In the case of a centered stationary Gaussian process, $X_t$, the Hilbert space can be chosen in the following way:\\
Consider the inner product $\Hprod{1_{[0,t]}}{1_{[0,s]}}=\text{Cov}[X_tX_s]=C(t-s)$ and take the Hilbert space $\hilbert$ as the closure of the set of step functions on $\RR$ with respect to this inner product. With this Hilbert space one concludes that $X_t=W(1_{[0,t]})$ and $\boldsymbol{D}X_t=1_{[0,t]}$ where $\boldsymbol{D}$ is the Malliavin derivative.
\end{preremark}
Since the Malliavin derivative satisfies the chain rule, we have $\boldsymbol{D}f(F)=f'(F)\boldsymbol{D}F$, for any $f:\RR\rightarrow\RR$ of class $\mathcal{C}^1$ with bounded derivative (which is also true for functions which are only a.e. differentiable, but with the assumption that $F$ is absolutely continous). The iterated Malliavin derivative, denoted by $\boldsymbol{D}^i$, can be define recursively. For $k\geq1$ and $p\geq 1$, $\mathbb{D}^{k,p}$ denotes the closure of $\mathcal{S}$ with respect to the norm $\left\|\cdot\right\|_{k,p}$ defined by
$$\left\|F\right\|_{k,p}^p=\mean{\abs{F}^p}+\sum_{i=1}^k\mean{\left\|\boldsymbol{D}^iF\right\|^p_{\hilbert^{\otimes i}}}.$$
Consider now an orthonormal system in $\hilbert$ denoted by $\{e_k\mid k\geq1\}$. Then, given elements $\phi\in\hilbert^{\otimes k_1},\psi\in\hilbert^{\otimes k_2}$, the contraction of order $r\leq\min\{k_1,k_2\}$ is the element of $\hilbert^{\otimes (k_1+k_2-2r)}$ defined by
$$\phi\otimes_r\psi=\sum_{i_1,\dots,i_r=1}^\infty \left\langle \phi,e_{i_1}\otimes\cdots e_{i_r}\right\rangle_{\hilbert^{\otimes r}}\left\langle \psi,e_{i_1}\otimes\cdots e_{i_r}\right\rangle_{\hilbert^{\otimes r}}.$$
In particular, $\phi\otimes_r\psi=\left\langle \phi,\psi\right\rangle_{\hilbert^{\otimes r}}$ when $k_1=k_2=r$.
\begin{preremark}\label{prodwiener}
Again, in the white noise framework (when $\hilbert=\Lspa^2_\mu$), for symmetric functions $\phi\in \Lspa^2_{\mu^{
\otimes k_1}}$, $\psi\in \Lspa^2_{\mu^{\otimes k_2}}$, the contraction is given by the integration of the first $r$ variables, i.e., $\phi\otimes_r\psi=\Lqprod{\phi}{\psi}{\mu^{\otimes r}}$. Also, we have a formula for the product of stochastic integrals:\footnote{$a\wedge b = \min\{a,b\}$ and $a\vee b = \max\{a,b\}$} $I_p(f)I_q(g)=\sum_{r=0}^{p\wedge q}r! {p \choose r}{q \choose r}I_{p+q-2r}(f\otimes_r g)$.
\end{preremark}

Define the divergence operator $\delta$ as the adjoint of the operator $\boldsymbol{D}$, so if $F\in\Dom\delta$ then $\delta(F)\in \Lspa^2(\Omega)$ and $\mean{\delta(F)G}=\mean{\Hprod{\boldsymbol{D}G}{F}}$.
\begin{preremark}\label{skorohod1}
When $\hilbert=\Lspa^2_\mu$, the divergence operator $\delta$ is called the Skorohod integral. It is a creation operator in the sense that for all $F\in\Dom\delta\subset \Lspa^2_{\mu\times\mathbb{P}}(T\times\Omega)$ with chaos representation $F(t)=\sum_{q=0}^\infty I_q\bigl(f_q(t,\cdot)\bigr)$ ($f_q\in \Lspa^2_{\mu^{\otimes(q+1)}}$ are symmetric functions in the last $q$ variables), $\delta(F)=\sum_{q=0}^\infty I_{q+1}(\widetilde{f}_q)$.\footnote{$\widetilde{f}$ is the symmetrization of $f$, i.e., $\tilde{f}(z_1,\dots,z_q)=\frac{1}{q!}\sum_\sigma f(z_{\sigma(1)},\dots,z_{\sigma(q)})$}
\end{preremark}
For all $F\in \Lspa^2(\Omega)$ denote by $J_qF$ the projection of $F$ in the $q^{\text{th}}$ chaos. Then, the Ornstein-Uhlenbeck semigroup is the family of contraction operators $\{T_t\ \mid t\geq0\}$ on $\Lspa^2(\Omega)$ defined by $T_tF=\sum_{q=0}^\infty e^{-qt}J_qF$. Using Mehler's formula we can find an equivalence between Mehler's semigroup and the O-U semigroup. More formally, take $W'$ as an independent copy of $W$ defining $(W,W')$ on the product probability space $(\Omega\times\Omega,\mathcal{F}\otimes\mathcal{F}',\mathbb{P}\times\mathbb{P})$. Each $F\in \Lspa^2(\Omega)$ can be regarded as measurable map $F(W)$ from $\RR^{\hilbert}$ to $\RR$ determined $\mathbb{P}\circ W^{-1}$-a.s. such that $T_tF=\mathbb{E}'\bigl[F\bigl(e^{-t}W+\sqrt{1-e^{-2t}}W'\bigr)\bigr]$. The infinitesimal generator for this semigroup (denoted by $L$) is given by $LF=\sum_{q=0}^\infty-qJ_qF$ and $\Dom L=\mathbb{D}^{2,2}= \Dom (\delta\boldsymbol{D})$. It can be proved that, for $F\in\Dom L$, $\delta\boldsymbol{D}F=-LF$. The pseudo-inverse of this operator (denoted by $L^{-1}$) is given by $L^{-1}F=\sum_{q=1}^\infty\frac{-1}{q}J_qF$, and is such that $L^{-1}F\in\Dom L$ and $LL^{-1}F=F-\mean{F}$ for any $F\in \Lspa^2(\Omega)$.

The use of Hermite polynomials is extremely important in this setting. The relationship between Hermite polynomials and Gaussian random variables is the following. Let $Z_1,Z_2$ be two random variables with joint Gaussian distribution such that $\mean{Z_1^2}=\mean{Z_2^2}=1$ and $\mean{Z_1}=\mean{Z_2}=0$. Then for all $q,p\geq0$,
\begin{align}\label{normalhermite}
\mean{H_p(Z_1)H_q(Z_2)}=
\begin{cases}
q!\bigl(\mean{Z_1Z_2}\bigr)^q & \text{if } q=p\\
0 & \text{if } q\neq p
\end{cases}.
\end{align}

On the other hand, one can expand a $\mathcal{C}^2$ function $f:\RR\rightarrow\RR$ in terms of Hermite polynomials, that is,
\begin{align}\label{decomp}
f(x)=\mean{f(Z)}+\sum_{q=1}^\infty c_qH_q(x),
\end{align}
where the real numbers $c_q$ are given by $c_qq!=\mean{f(Z)H_q(Z)}$ and $Z\sim\mathcal{N}(0,1)$.
\begin{preremark}\label{hermiteintegral}
Notice that in the white noise case we have the relationship $H_q\bigl(W(h)\bigr)=\int_{T^q}h^{\otimes q}dW_{t_1}\cdots dW_{t_q}=I_q\bigl(h^{\otimes q}\bigr)$ so the decomposition (\ref{decomp}) of $f(W(h))$ can be regarded as
$$f(W(h))=\sum_{q=0}^\infty c_qI_q\bigl(h^{\otimes q}\bigr).$$
\end{preremark}

With (\ref{normalhermite}) and (\ref{decomp}) we are able to compute the covariance of a real function $f$ in the following way,
\begin{align}\label{covf}
\text{Cov}[f(Z_1),f(Z_2)]=\sum_{p,q=1}^\infty c_pc_q\mean{H_p(Z_1)H_q(Z_2)}=\sum_{q=1}^\infty c_q^2q!\bigl(\mean{Z_1Z_2}\bigr)^q.
\end{align}

\subsection{Malliavin calculus on Wiener-Poisson space}

Let $\hilbert=\Lspa^2_\mu$. Assume there is a complete probability space $(\Omega,\mathcal{F},\mathbb{P})$ where $\Levyp_t$ is a cadlag, centered, L\'evy process: $\Levyp_t$ has stationary and independent increments, is continuous in probability and $\Levyp_0=0$, with $\mean{\Levyp_1^2}<\infty$. At the risk of causing some confusion, $\mathcal{F}$ denotes the filtration generated by $\Levyp_t$ completed with the null sets of the above filtration, and work on the space $(\Omega,\mathcal{F},\mathbb{P})$. Assume this process is represented by the triplet $(0,\sigma^2,\nu)$, where $\nu$ is the L\'evy measure such that $d\mu(t,x)=\sigma^2dt\delta_0(x)+x^2dtd\nu(x)$ and $\int_\RR x^2\,d\nu(x)<\infty$. This process can be represented as
$$\Levyp_t=\sigma W_t + \int_0^t\int_{\RR} x\,d\widetilde{N}(s,x),$$
where $W_t$ is a standard Brownian motion, $\sigma\geq0$ and $\widetilde{N}$ is the compensated jump measure. A fuller exposition on L\'evy processes can be found in \cite{Applebaum} and \cite{Sato}. This process is extended to a random measure $M$, which is used to construct (in an analogous way to the It$\hat{\text{o}}$ integral construction) an integral on the step functions, and then by linearity and continuity it is extended to $\Lspa^2\bigl((\RR^+\times\RR)^q,\mathcal{B}(\RR^+\times\RR)^q,\mu^{\otimes q}\bigr)$ and denoted by $I_q$. This integral satisfy the following properties:
\begin{enumerate}

\item $I_q(f)=I_q(\tilde{f})$.

\item $I_q(af+bg)=aI_q(f)+bI_q(g)$ \hspace{1cm} ($a,b\in\RR$).

\item $\mean{I_p(f)I_q(g)}=q!\int_{(\RR^+\times\RR)^q}\tilde{f}\tilde{g}\,d\mu^{\otimes q}1_{\{q=p\}}$.

\end{enumerate}
These properties are stated in \cite{Josep} and their proof can be found in \cite{Ito}. We have a product formula in this framework which is similar to the one in Remark \ref{prodwiener} but with extra terms coming from the Poisson integration part. A product formula for the pure jump framework can be found in \cite{Peccati}. Before stating the formula, we need to define a general version of the contraction. Let $\phi\in \Lspa^2_{\mu^{\otimes k_1}}$ and $\psi\in \Lspa^2_{\mu^{\otimes k_2}}$ be symmetric functions. Then the general contraction of order $r\leq\min\{k_1,k_2\}$ and $s\leq\min\{k_1,k_2\}-r$ is given by the integration of the first $r$ variables and the ``sharing'' of the following $s$ variables, i.e., $\phi\otimes_r^s\psi=\prod_{i=1}^sz_{2i}\Lqprod{\phi(\cdot,z,x)}{\psi(\cdot,z,y)}{\mu^{\otimes r}}$, where $z\in(\RR^2)^s$ and $(x,y)\in(\RR^2)^{k_1-r-s}\times(\RR^2)^{k_2-r-s}$. Now the product formula can be stated as follows. If $\abs{f}\otimes_r^s\abs{g}\in \Lspa^2_{\mu^{p+q-2r-s}}$ for $0\leq r\leq \min\{p,q\}$ and $0\leq s\leq \min\{p,q\}-r$, then
\begin{align}\label{prodint}
I_p(f)I_q(g)=\sum_{r=0}^{p\wedge q}\sum_{s=0}^{p\wedge q-r}r!s! {p \choose r}{q \choose r}{p-r \choose s}{q-r \choose s}I_{p+q-2r-s}(f\otimes_r^s g).
\end{align}
The proof of this product formula can be found in \cite{Lee}.
\begin{preremark}
In the general contraction formula $z_{2i}$ is the size of the jump and $z_{2i-1}$ is the time when that jump occurs $\bigl(z=(z_1,z_2,\dots,z_{2s-1},z_{2s})\bigr)$. If we only have the Wiener part, the factor $z_{2i}$ would be zero unless $s=0$, and then we obtain the contraction defined in the Wiener space. Similarly, when the terms where $s\neq0$ are zero, the formula (\ref{prodint}) reduces to that in Remark \ref{prodwiener}
\end{preremark}
It$\hat{\text{o}}$ has also proved \cite[Theorem 2]{Ito} that for all $F\in \Lspa^2(\Omega):=\Lspa^2(\Omega,\mathcal{F},\mathbb{P})$, we have,
\begin{align}\label{chaosrep1}
F=\sum_{q=0}^\infty I_q(f_q),\hspace{1cm} f_q\in \Lspa^2_{\mu^{\otimes q}}:=\Lspa^2\bigl(\RR^+\times\RR)^q,\mathcal{B}(\RR^+\times\RR)^q,\mu^{\otimes q}\bigr),
\end{align}
and that this representation is unique if the $f_q$'s are symmetric functions. From this chaotic representation we can define the annihilation operators and creation operators, the former will be the Malliavin derivatives and the latter will be the Skorohod integrals. In this way define 
$\Dom\boldsymbol{D}$ as the set of functionals $F\in \Lspa^2(\Omega)$ represented as in (\ref{chaosrep1}) such that $\sum_{q=1}^\infty qq!\Lqnorm{f_q}{\mu^{\otimes q}}^2<\infty$. For $F\in\Dom\boldsymbol{D}$ the Malliavin derivative of F is the stochastic process given by 
\begin{align}\label{malliavinderivative2}
\boldsymbol{D}_zF=\sum_{q=0}^\infty qI_{q-1}(f_q(z,\cdot)),\hspace{1cm} z\in\RR^+\times\RR,\ f_q \text{ symmetric}.
\end{align}
If we define the inner product as $\Lqprod{f}{g}{\mu}=\int_{\RR^+\times\RR}f(z)g(z) d\mu(z)$, then $\Dom\boldsymbol{D}$ is a Hilbert space with the inner product $\left\langle F,G\right\rangle=\mean{FG}+\mean{\Lqprod{\boldsymbol{D}_zF}{\boldsymbol{D}_zG}{\mu}}$. We can embed 
$\Dom\boldsymbol{D}$ in two spaces $\Dom\boldsymbol{D}^0$ and $\Dom\boldsymbol{D}^J$. $\Dom\boldsymbol{D}^0$ is defined as the set of all functionals $F\in \Lspa^2(\Omega)$ with representation given as in (\ref{chaosrep1}) such that $\sum_{q=1}^\infty qq!\int_{\RR^+}\Lqnorm{f_q\bigl((t,0),\cdot\bigr)}{\mu^{\otimes(q-1)}}^2\hspace{-.8cm}dt<\infty$, while $\Dom\boldsymbol{D}^J$ is defined as the respective functionals satisfying $\sum_{q=1}^\infty qq!\int_{\RR^+\times\RR_0}\Lqnorm{f_q(z,\cdot)}{\mu^{\otimes(q-1)}}^2\hspace{-.8cm}d\mu(z)<\infty$; hence 
$\Dom\boldsymbol{D} = \Dom\boldsymbol{D}^0\cap\Dom\boldsymbol{D}^J$. We can now rewrite (due to the independence of $W$ and $\widetilde{N}$) $\Omega$ as the cross product $\Omega_W\times\Omega_{\widetilde{N}}$.
\begin{itemize}

\item {\bf Derivative $\boldsymbol{D}_{t,0}$}\\
This derivative can be interpreted as the derivative with respect to the Brownian motion part. Using the isometry $\Lspa^2(\Omega)\simeq \Lspa^2\bigl(\Omega_W;\Lspa^2(\Omega_{\widetilde{N}})\bigr)$ (with $\Omega=\Omega_W\times\Omega_{\widetilde{N}}$), we can define a Malliavin derivative as we did in the Wiener case but using the $\Lspa^2(\Omega_{\widetilde{N}})$-valued smooth random variables $S_{\widetilde{N}}$, that is, for the functionals of the form $F=\sum_{i=1}^nG_iH_i$, where $G_i\in \mathcal{S}$ and $H_i\in \Lspa^2(\Omega_{\widetilde{N}})$. Then, this derivative will be $\boldsymbol{D}^WF=\sum_{i=1}^n(\boldsymbol{D}^WG_i)H_i$ and this $\boldsymbol{D}^WG_i$ is the derivative defined in (\ref{malliavinderivative}). This definition is extended (see \cite{Josep}) to a subspace 
$\Dom\boldsymbol{D}^W\subset \Dom\boldsymbol{D}^0$ and for $F\in\Dom\boldsymbol{D}^W$,
\begin{align}\label{derivative0}
\boldsymbol{D}_{t,0}F=\frac{1}{\sigma}\boldsymbol{D}_t^WF.
\end{align}
Furthermore, we also have a chain rule result for functionals of the form $F=f(G,H)\in \Lspa^2(\Omega)$ with $G\in\Dom\boldsymbol{D}^W$, $H\in \Lspa^2(\Omega_{\widetilde{N}})$ and $f(x,y)$ continously differentiable in the variable $x$ with bounded partial derivative. We conclude that 
$F\in\Dom\boldsymbol{D}^0$ and
\begin{align}\label{chainrule0}
\boldsymbol{D}_{t,0}F=\frac{1}{\sigma}\frac{\partial f}{\partial x}(G,H)\boldsymbol{D}_t^WG.
\end{align}
This is also true (as in the Wiener space case) for functions which are a.e. differentiable but with the restriction that $G$ satisfies an absolutely continuous law.

\item {\bf Derivative $\boldsymbol{D}_z$ ($z\neq(t,0)$)}\\
This derivate has been shown to be a quotient operator $\Psi_{t,x}$, that is, if $\mean{\dint_{\RR^+\times\RR_0}(\Psi_zF)^2d\mu(z)}<\infty$ then $F\in\Dom\boldsymbol{D}^J$ and the Malliavin derivative with $z=(t,x)\neq(t,0)$ will be given by
\begin{align}\label{derivative1}
\boldsymbol{D}_zF=\Psi_{t,x}F:=\frac{F(\omega_{t,x})-F(\omega)}{x}.
\end{align}
The idea is to introduce a jump of size $x$ at moment $t$. See \cite{Josep0} for a complete contruction on the canonical space in which this is developed, and \cite{Josep} for a quick explanation on how to introduce a jump at moment $t$ and the conditions on the $\omega$'s.

\end{itemize}

\bigskip

If $FG\in \Dom\MDer^W\cap\Dom\MDer^J$ such that $\mean{\dint_{\RR^+\times\RR_0}\bigl(\MDer_z(FG)\bigr)^2\,d\mu(z)}<\infty$, then (using the fact that the jump $x$ is zero for $\boldsymbol{D}_{t,0}$) it is easy to conclude by direct calculation that (see \cite[Proposition 5.1.]{Josep0} and \cite[Exercise 1.2.12]{Nualart}),
\begin{align}\label{productrule1}
\boldsymbol{D}_z(FG)=\boldsymbol{D}_zF\cdot G+F\cdot\boldsymbol{D}_zG+x\cdot\boldsymbol{D}_zF\cdot\boldsymbol{D}_zG,
\end{align}
for all $z=(t,x)\in\RR^+\times\RR$. Moreover, if $\Hprod{\MDer F}{\MDer G}\in\Dom\MDer^{W,J}$, the following product rule also holds,
\begin{align}\label{productrule2}
\boldsymbol{D}_z\Hprod{\MDer F}{\MDer G}=\Hprod{\MDer^2_z F}{\MDer G}+\Hprod{\MDer F}{\MDer^2_z G}+\Hprod{x\MDer F^2_z}{\MDer^2_z G}.
\end{align}
The case $\nu\equiv0$ was shown in \cite[Lemma 3.2]{Nourdin}, and the case $\nu\not\equiv0$ follows directly from
\begin{align*}
\boldsymbol{D}_z\Hprod{\MDer F}{\MDer G}&=\boldsymbol{D}_z\int_{\RR^+\times\RR}\MDer_uF\MDer_uG\,d\mu(u)=\int_{\RR^+\times\RR}\MDer_z(\MDer_uF\MDer_uG)\,d\mu(u)\\
&=\int_{\RR^+\times\RR}\boldsymbol{D}^2_{z,u}F\cdot \MDer_uG+\MDer_uF\cdot\boldsymbol{D}^2_{z,u}G+x\cdot\boldsymbol{D}^2_{z,u}F\cdot\boldsymbol{D}^2_{z,u}G\,d\mu(u).
\end{align*}

\bigskip

In the same way, consider the chaotic decomposition $F(z)=\sum_{q=0}^\infty I_q\bigl(f_q(z,\cdot)\bigr)$, with $f\in \Lspa^2_{\mu^{\otimes q}}$ symmetric with respect to the last $q$ variables. If $\sum_{q=0}^\infty (q+1)!\Lqnorm{f_q}{\mu^{\otimes(q+1)}}^2\hspace{-.7cm}<\infty$ then we say that $F\in\Dom\delta$. Now we can define the Skorohod integral of $F\in\Dom\delta$ by
\begin{align}\label{skorohod2}
\delta(F)=\sum_{q=0}^\infty I_{q+1}(\widetilde{f}_q)\in \Lspa^2(\Omega).
\end{align}
This operator is the adjoint of the operator $\boldsymbol{D}_z$, so $\mean{\delta(F)G}=\mean{\Lqprod{F(z)}{\boldsymbol{D}_zG}{\mu}}$ for all $G\in\Dom\boldsymbol{D}$. Denote by $\mathbb{L}^{1,2}$ the set of elements $F\in \Lspa^2_{\mu\otimes\mathbb{P}}\bigl(\RR^+\times\RR\times\Omega\bigr)$ such that $\sum_{q=1}^\infty qq!\Lqnorm{f_q}{\mu^{\otimes q}}^2\hspace{-.4cm}<\infty$. For all $F\in\mathbb{L}^{1,2}\subset\Dom\delta$ we have that $F(z)\in\Dom\boldsymbol{D}$, $\forall \ z \ \mu-a.e.$ and that $\boldsymbol{D}_\cdot F(\cdot)\in \Lspa^2_{\mu^{\otimes2}\times\mathbb{P}}\bigl((\RR^+\times\RR)^2\times\Omega\bigr)$.

\bigskip

Finally, the definitions of the Ornstein-Uhlenbeck semigroup $T_t$ and its infinitesimal generator $L$ are parallel to the ones in the Wiener space case. Basically, all we need to define is the Malliavin derivative and the Skorohod integral, that is, we can just define $L:=-\delta\boldsymbol{D}$. With this definition we obtain that for $F\in \Lspa^2(\Omega)$ with chaotic representation (\ref{chaosrep1}), $LF=\sum_{q=1}^\infty -qI_q(f_q)$ and $T_tF=\sum_{q=0}^\infty e^{-qt}I_q(f_q)$. In the same way, the pseudo-inverse is given by $L^{-1}F=\sum_{q=1}^\infty\frac{-1}{q}I_q(f_q)$ and $LL^{-1}F=F-E[F].$

\section{Main theorems}

The first tool needed is a generalized version of the chain rule in the framework of the Wiener-Poisson space. But first we need to define a suitable subset of $\Dom\boldsymbol{D}$ where \eqref{derivative0} and \eqref{derivative1} remains valid and the chain rule for a general random variable in the Wiener-Poisson space can be implemented. Consider the set
$$\Ssmooth^{W,J}:=\left\{\left.\sum_{i=1}^nG_iH_i\ \right|\ G_i\in \Ssmooth; H_i\in \Dom\boldsymbol{D}^J \text{ such that } \mean{\dint_{\RR^+\times\RR_0}(\Psi_zH_i)^2\,d\mu(z)}<\infty\right\}.$$
With the inner product given by
$$\left\langle F,G\right\rangle_{W,J}:=\mean{FG}+\mean{\Lqprod{\boldsymbol{D}_zF}{\boldsymbol{D}_zG}{\mu}}=\meanW{\meanJ{FG}}+\meanW{\meanJ{\Lqprod{\boldsymbol{D}_zF}{\boldsymbol{D}_zG}{\mu}}},$$
let $\Dom\MDer^{W,J}$ be the closure of $\Ssmooth^{W,J}$ with respect to the implicit norm (see \cite[Remarks 2 and 3 - page 31]{Nualart} for the properties of this space). Clearly, $\Dom\MDer^{W,J}\subset\Dom\MDer^W\cap\Dom\MDer^J\subset\Dom\boldsymbol{D}$, and for all $F\in\Dom\MDer^{W,J}$ formulas \eqref{derivative0} and \eqref{derivative1} hold. In the following proposition, for an $\RR^d$-valued random variable $F=\bigl(F^{(1)},\dots,F^{(d)}\bigr)$, $\MDer_zF$ must be understood as the vector $\bigl(\MDer_zF^{(1)},\dots,\MDer_zF^{(d)}\bigr)$.

\begin{prop}\label{chainrulemultvar}
(General Chain Rules)\\
Fix $k\geq2$. Let $f:\RR^d\to\RR$ be a $\mathcal{C}^{k-1}(\RR^d)$ with bounded gradient, that is, $\bigl\|\abs{\nabla f}_{\RR^d}\bigr\|_\infty<\infty$, such that $\partial^\alpha f$ is a.e. differentiable for any multi-index $\alpha$ such that $\abs{\alpha}=k-1$. Let the random vector $F=\bigl(F^{(1)},\dots,F^{(d)}\bigr)$ be such that $F^{(i)}\in\Dom\MDer^{W,J}$ for all $i$. Then $f(F)\in\Dom\MDer^{W,J}$ and
\begin{align}\label{chainrulemultvar1f}
\MDer_zf(F)=&\sum_{\abs{\alpha}=1}^{k-1}\frac{\partial^\alpha f(F)}{\alpha!}x^{\abs{\alpha}-1}(\MDer_z F)^\alpha\\
&\quad+\frac{1}{k!}\sum_{i_1,\dots,i_k=1}^dx^{k-1}\prod_{r=1}^k(\MDer_z F^{(i_r)})\int_0^1(1-t)^{k-1}\partial_{i_1,\dots,i_k}f(F+tx\MDer_zF)\,dt.\nonumber
\end{align}
Furthermore, if $f\in\mathcal{C}^{k}(\RR^d)$ then
\begin{align}\label{chainrulemultvar2f}
\MDer_zf(F)=\sum_{\abs{\alpha}=1}^{k-1}\frac{\partial^\alpha f(F)}{\alpha!}x^{\abs{\alpha}-1}(\MDer_z F)^\alpha+\sum_{\abs{\alpha}=k}\frac{\partial^\alpha f(F+\theta_zx\MDer_zF)}{\alpha!}x^{k-1}(\MDer_z F)^\alpha,
\end{align}
for some function $\theta_z\in(0,1)$ depending on $z$, $F$ and $\MDer_zF$.
\end{prop}
\begin{proof}
This will be proven for the derivatives in the Wiener direction and the Poisson direction.\\

{\bf $\boldsymbol{\bullet\ }$ Case $\boldsymbol{z=(t,0):}$}\\
Consider first $F=\left(F^{(1)},\dots,F^{(d)}\right)$, where $F^{(i)}=\sum_{k=1}^{N_i}G_kH_k\in\Ssmooth^{W,J}$ for all $i$. Let $\widetilde{f}:\RR^{N_1+\cdots+N_d}\to\RR$ be such that $f(F)=\widetilde{f}(G)$, where $G=\left(G_1^{(1)},\dots,G_{N_1}^{(1)},\dots,G_1^{(d)},\dots,G_{N_{d}}^{(d)}\right)$. By formula \eqref{derivative0} and the chain rule in the Wiener space,
\begin{align*}
\MDer_{t,0}f(F)&=\frac{1}{\sigma}\MDer^W_t\widetilde{f}(G)=\frac{1}{\sigma}\sum_{i=1}^d\sum_{k=1}^{N_i}\frac{\partial \widetilde{f}}{\partial x_{i+k}}(G)\MDer^WG^{(i)}_k=\sum_{i=1}^d\sum_{k=1}^{N_i}\frac{\partial \widetilde{f}}{\partial x_{i+k}}(G)\MDer_{t,0}G^{(i)}_k.
\end{align*}
On the other hand, $\frac{\partial\widetilde{f}}{\partial x_{i+k}}(G)=\frac{\partial f}{\partial x_i}(F)H^{(i)}_k$. Hence,

\vspace{-.7cm}

$$\MDer_{t,0}f(F)=\sum_{i=1}^d\sum_{k=1}^{N_i}\frac{\partial f}{\partial x_i}(F)H^{(i)}_k\MDer_{t,0}G^{(i)}_k=\sum_{i=1}^d\frac{\partial f}{\partial x_i}(F)\overbrace{\MDer_{t,0}\sum_{k=1}^{N_i}H^{(i)}_kG^{(i)}_k}^{=\ \MDer_{t,0}F^{(i)}}.$$
Accordingly, $\MDer_{t,0}f(F)=\sum_{i=1}^d\frac{\partial f}{\partial x_i}(F)\MDer_{t,0}F^{(i)}$.

\bigskip

Consider now a general $F\in\Dom\MDer^{W,J}$, and take a sequence $\{F_N\}\subset\Ssmooth^{W,J}\times\cdots\times\Ssmooth^{W,J}$ converging in $\|\cdot\|_{W,J}$ to $F$, that is, $F^{(i)}_N\xrightarrow[N\to\infty]{\Lspa^2}F^{(i)}$ and $\MDer F^{(i)}_N\xrightarrow[N\to\infty]{\Lspa^2}\MDer F^{(i)}$ in the respective spaces. By the boundedness and continuity of $\nabla f$, along with the mean value theorem,
\begin{align*}
\left\| f(F)-f(F_N) \right\|_{\mathbb{L}^{2}(\Omega)}&=\left\| \nabla f\bigl(\alpha_{F,F_N}F+(1-\alpha_{F,F_N})F_N)\bigr)\cdot(F-F_N)\right\|_{\mathbb{L}^{2}(\Omega)}\\
&\leq \bigl\|\abs{\nabla f}_{\RR^d}\bigr\|_\infty\bigl\| \abs{F-F_N}_{\RR^d}\bigr\|_{\mathbb{L}^{2}(\Omega)}\leq \bigl\|\abs{\nabla f}_{\RR^d}\bigr\|_\infty\sum_{i=1}^d\left\| F^{(i)}-F^{(i)}_N\right\|_{\mathbb{L}^{2}(\Omega)}\arrowlim{N}{\infty}0.
\end{align*}
Since $\left\|\sum_{i=1}^d\frac{\partial f}{\partial x_i}(F_N)\MDer_{t,0}F^{(i)}_N\right\|_{\mathbb{L}^{2}(\Omega\times\RR^+)}\leq\bigl\|\abs{\nabla f}_{\RR^d}\bigr\|_\infty\left(\sum_{i=1}^d\mean{\left\| \MDer_{t,0}F^{(i)}\right\|_{\mathbb{L}^{2}(\RR^+)}}+K\right)<\infty$, for some constant $K$, then by the completeness of the space, a convergent subsequence exists. It is known from \cite[Remark 2 - page 31]{Nualart} that $\MDer_{t,0}$ is a closed operator from $\Dom\MDer^{W,J}\subset\Lspa^2\bigl(\Omega_W;\Lspa^2(\Omega_{\widetilde{N}})\bigr)\simeq \Lspa^2(\Omega_W\times\Omega_{\widetilde{N}})\simeq \Lspa^2(\Omega)$ into $\Lspa^2\bigl(\Omega_W;\Lspa^2(\RR^+,\mathcal{B}(\RR^+),dt)\otimes\Lspa^2(\Omega_{\widetilde{N}})\bigr)\simeq\Lspa^2\bigl(\Omega\times\RR^+,\mathcal{F}\otimes\mathcal{B}(\RR^+),\mathbb{P}\otimes dt\bigr)$. Therefore, $f(F)\in\Dom\MDer^W$.

\bigskip

Finally, consider the set $\Ssmooth^{W,J}_{\Lspa^2}:=\left\{\sum_{i=1}^nf_i(t)G_iH_i\ \mid G_i\in \Ssmooth; H_i\in \Dom\boldsymbol{D}^J; f_i\in\Lspa^2(\RR^+)\right\}$. By \cite[page 37]{Nualart}, it is known that $\Ssmooth^{W,J}_{\Lspa^2}\in \Dom\delta^W$. Hence, for all $G\in\Ssmooth^{W,J}_{\Lspa^2}$,
\begin{align*}
\mean{\left\langle \MDer f(F),G \right\rangle_{\Lspa^2(\Omega_{\widetilde{N}}\times\RR^+)}}&=\mean{f(F)\delta^W(G)}=\lim_{l\to\infty}\mean{f\bigl(F_{N_l}\bigr)\delta(G)}=\lim_{l\to\infty}\mean{\left\langle \MDer f\bigl(F_{N_l}\bigr),G \right\rangle_{\Lspa^2(\Omega_{\widetilde{N}}\times\RR^+)}}\\
&=\lim_{l\to\infty}\mean{\left\langle \nabla f\bigl(F_{N_l}\bigr)\cdot\MDer F_{N_l},G \right\rangle_{\Lspa^2(\Omega_{\widetilde{N}}\times\RR^+)}}=\mean{\left\langle \nabla f(F)\cdot\MDer F,G \right\rangle_{\Lspa^2(\Omega_{\widetilde{N}}\times\RR^+)}}\\
&=\mean{\left\langle \sum_{i=1}^d\frac{\partial f}{\partial x_i}(F)\MDer F^{(i)},G \right\rangle_{\Lspa^2(\Omega_{\widetilde{N}}\times\RR^+)}}.
\end{align*}
This shows that $\mean{\left\langle \MDer f(F)-\sum_{i=1}^d\frac{\partial f}{\partial x_i}(F)\MDer F^{(i)},G \right\rangle_{\Lspa^2(\Omega_{\widetilde{N}}\times\RR^+)}}=0$ for any $G\in\Ssmooth^{W,J}_{\Lspa^2}$. Since $\Ssmooth^{W,J}_{\Lspa^2}$ is dense in $\Lspa^2\bigl(\Omega_W;\Lspa^2(\RR^+,\mathcal{B}(\RR^+),dt)\otimes\Lspa^2(\Omega_{\widetilde{N}})\bigr)$, and $z=(t,0)$ it follows that
\begin{align*}
\MDer f(F)&= \sum_{i=1}^d\frac{\partial f}{\partial x_i}(F)\MDer F^{(i)}\\
&=\sum_{\abs{\alpha}=1}^{k-1}\frac{\partial^\alpha f(F)}{\alpha!}0^{\abs{\alpha}-1}(\MDer_z F)^\alpha+\frac{1}{k!}\sum_{i_1,\dots,i_k=1}^d0^{k-1}\prod_{r=1}^k(\MDer_z F^{(i_r)})\int_0^1(1-t)^{k-1}\partial_{i_1,\dots,i_k}f(F+t0\MDer_zF)\,dt.
\end{align*}
The convention $0^0=1$ is used in the last equality. The proof of (\ref{chainrulemultvar2f}) is analogous.

\bigskip

{\bf $\boldsymbol{\bullet\ }$ Case $\boldsymbol{z=(t,x)}$ with $\boldsymbol{x\neq0}$:}\\
Since $f$ is differentiable, then by the mean value theorem we obtain
\begin{align*}
\abs{\Psi_zf(F)}:&=\abs{\frac{f\bigl(F(\omega_{t,x})\bigr)-f\bigl(F(\omega)\bigr)}{x}}=\abs{\nabla f\bigl(\alpha_{F,z}F(\omega_{t,x})+(1-\alpha_{F,z})F)\bigr)\cdot\boldsymbol{D}_zF}\\
&\leq \abs{\nabla f\bigl(\alpha_{F,z}F(\omega_{t,x})+(1-\alpha_{F,z})F\bigr)}_{\RR^d}\sum_{i=1}^d\abs{\boldsymbol{D}_zF^{(i)}}_{\RR^d}.
\end{align*}
Therefore,
$$\mean{\dint_{\RR^+\times\RR_0}\bigl(\Psi_zf(F)\bigr)^2\,d\mu(z)}\leq \bigl\|\abs{\nabla f}_{\RR^d}\bigr\|_\infty\left(\sum_{i=1}^d\mean{\Lqnormmu{\boldsymbol{D}_zF^{(i)}}{\mu}^2}^{\frac{1}{2}}\right)^2 <\infty.$$
Hence, $f(F)\in\Dom\MDer^J$ and the Malliavin derivative is given by $\MDer_zf(F)=\Psi_zf(F)$. By Taylor's formula,
\begin{align*}
f(y)=&f(y_0)+\sum_{\abs{\alpha}=1}^{k-1}\frac{\partial^\alpha f(y_0)}{\alpha!}(y-y_0)^\alpha+\frac{1}{k!}\sum_{i_1,\dots,i_k=1}^d\prod_{r=1}^k\bigl(y^{(i_r)}-y_0^{(i_r)}\bigr)\int_0^1(1-t)^{k-1}\partial_{i_1,\dots,i_k} f(y_0+t(y-y_0))\,dt.
\end{align*}
Recalling that $y^{(i)}-y^{(i)}_0=F^{(i)}(\omega_z)-F^{(i)}(\omega)=x\boldsymbol{D}_zF^{(i)}$, if $y=F(\omega_z)$ and $y_0=F(\omega)$, and plugging in these values, the formula (\ref{chainrulemultvar1f}) is immediately verified.\\
The proof of (\ref{chainrulemultvar2f}) follows the same logic but with the following Taylor's formula:
$$f(y)=\sum_{\abs{\alpha}=1}^{k-1}\frac{\partial^\alpha f(y_0)}{\alpha!}\bigl(y-y_0\bigr)^\alpha+\sum_{\abs{\alpha}=k}\frac{\partial^\alpha f(y_0+\theta(y-y_0))}{\alpha!}\bigl(y-y_0\bigr)^\alpha,$$
where $\theta\in(0,1)$ depends on $y_0$ and $y$.
\end{proof}

\bigskip

This chain rule allows us to employ the so-called {\it ``Nourdin-Peccati analysis''}, which is stated in the following proposition.

\begin{prop}\label{NPintbypartsWP}
Let the random vectors $Z=(Z^{(1)},\dots,Z^{(d)})$ and $F=(F^{(1)},\dots,F^{(d)})$ be such that $Z^{(i)},F^{(i)}\in\Dom\MDer^{W,J}$, for all $i\in\{1,\dots,d\}$. Suppose that the function $f:\RR^d\to\RR$ satisfies that $\nabla f\in \mathcal{C}^2(\RR^d)$ with a bounded Hessian. Then,
\begin{align}\label{NPintbypartsformulaWP}
\mean{\Rinnprod{Z}{\nabla f(F)}{d}}&=\Rinnprod{\mean{Z}}{\mean{\nabla f(F)}}{d}+\mean{\HSinnprod{\Hess_f(F)}{g_{Z,F}(F)}}\nonumber\\
&\quad+\sum_{j=1}^d\mean{\Hprod{\sum_{\abs{\beta}=2}\frac{\partial^\beta\partial_jf(F+\theta_zx\MDer_zF)x(\MDer_zF)^\beta}{\beta!}}{-\MDer L^{-1}\bigl(Z^{(j)}-\mathbb{E}\bigl[Z^{(j)}\bigr]\bigr)}},
\end{align}
where $g_{Z,F}$ is the matrix $g_{Z,F}^{i,j}(x)=\mean{\left.\Hprod{\MDer F^{(i)}}{-\MDer L^{-1}\bigl(Z^{(j)}-\mathbb{E}\bigl[Z^{(j)}\bigr]\bigr)}\ \right|\ F=x}$, $\HSinnprod{\cdot}{\cdot}$ is the Hilbert-Schmidt inner product, and $\hilbert$ is the underlying Hilbert space. If $\nu\not\equiv0$ (the jump part is present), then $\hilbert=\Lspa^2_\mu$.
\end{prop}
\begin{proof}
Notice that it is enough to consider $\nabla f(F)$ and $Z$ centered, because
$$\mean{\Rinnprod{Z-\mean{Z}}{\nabla f(F)-\mean{\nabla f(F)}}{d}}=\mean{\Rinnprod{\nabla f(F)}{Z}{d}}-\Rinnprod{\mean{\nabla f(F)}}{\mean{Z}}{d}.$$
Hence, let $\mean{f(F)}=0=\mean{Z}$. By the chain rule \eqref{chainrulemultvar2f},
\begin{align*}
\mean{\Rinnprod{Z}{\nabla f(F)}{d}}&=\sum_{j=1}^d\mean{Z^{(j)}\frac{\partial f}{\partial x_j}(F)}=\sum_{j=1}^d\mean{\frac{\partial f}{\partial x_j}(F)\left(LL^{-1}Z^{(j)}\right)}\\
&=\sum_{j=1}^d\mean{\frac{\partial f}{\partial x_j}(F)\left(-\delta\MDer L^{-1}Z^{(j)}\right)}=\sum_{j=1}^d\mean{\Hprod{\MDer\frac{\partial f}{\partial x_j}(F)}{-\MDer L^{-1}Z^{(j)}}}\\
&=\sum_{j=1}^d\sum_{i=1}^d\mean{\frac{\partial^2 f}{\partial x_j\partial x_i}(F)\mean{\left.\Hprod{\MDer F^{(i)}}{-\MDer L^{-1}Z^{(j)}}\right|F}}\\
&\quad\quad\quad+\sum_{j=1}^d\mean{\Hprod{\sum_{\abs{\beta}=2}\frac{\partial^\beta\partial_jf(F+\theta_zx\MDer_zF)x(\MDer_zF)^\beta}{\beta!}}{-\MDer L^{-1}Z^{(j)}}}\\
&=\mean{\HSinnprod{\Hess_f(F)}{g_{Z,F}(F)}}\\
&\quad\quad\quad+\sum_{j=1}^d\mean{\Hprod{\sum_{\abs{\beta}=2}\frac{\partial^\beta\partial_jf(F+\theta_zx\MDer_zF)x(\MDer_zF)^\beta}{\beta!}}{-\MDer L^{-1}Z^{(j)}}}.
\end{align*}
\end{proof}

\bigskip

Another important tool is the extension of the so-called {\em Gaussian Poincar\'e inequality} to the present context. But to prove this we need an inequality similar to the one proved in \cite[Proposition 3.1]{Nourdin} (was proved for all $p\geq2$ in the Wiener space case). The technique used in their proof was based on the equivalence between Mehler and Ornstein-Uhlenbeck semigroups for the Gaussian case, but in the Wiener-Poisson space we lack such an equivalence. Nevertheless, it is possible to prove it for $p=2$ and that is, in fact, the case needed to prove the extension of the Gaussian Poincar\'e inequality.
\begin{prop}\label{ineq1}
Let $F\in\Dom\boldsymbol{D}$ satisfy $\mean{F}=0$. Then,
$$\mean{\Hnorm{\boldsymbol{D}L^{-1}F}^2}\leq \mean{\Hnorm{\boldsymbol{D}F}^2}.$$
\end{prop}
\begin{proof}
Assume $F$ has its chaos decomposition given by (\ref{chaosrep1}). By the orthogonality between chaoses we get,
\begin{align*}
 \mean{\Hnorm{\boldsymbol{D}L^{-1}F}^2}=&\mean{\Hnorm{\sum_{q=1}^\infty\boldsymbol{D}\frac{1}{q}I_q(f_q)}^2}=\mean{\sum_{q=1}^\infty\frac{1}{q^2}\Hnorm{\boldsymbol{D}I_q(f_q)}^2}\leq \mean{\sum_{q=1}^\infty\Hnorm{\boldsymbol{D}I_q(f_q)}^2} =\mean{\Hnorm{\boldsymbol{D}F}^2}.
\end{align*}
\end{proof}

\begin{prop}\label{poincareineq}
(Extension of the Gaussian Poincar\'e inequality)\\
Let $F\in\Dom\boldsymbol{D}$. Then,
\begin{align}\label{poincareineq2}
\Var{F}\leq \mean{\Hnorm{\boldsymbol{D}F}^2},
\end{align}
with equality if and only if $F$ is a linear combination of elements in the first and $0^{\text{th}}$ chaos.
\end{prop}
\begin{proof}
Assume, without loss of generality, that $E[F]=0$.
\begin{align*}
\Var{F}=\mean{F^2}=\mean{\Hprod{\boldsymbol{D}F}{-\boldsymbol{D}L^{-1}F}}\leq \mean{\Hnorm{\boldsymbol{D}F}^2}^{\frac{1}{2}}\mean{\Hnorm{\boldsymbol{D}L^{-1}F}^2}^{\frac{1}{2}}\leq \mean{\Hnorm{\boldsymbol{D}F}^2},
\end{align*}
where Proposition \ref{ineq1} was used in the last step, and the fact that $F=-\delta\boldsymbol{D}L^{-1}F$ in the second step.
\end{proof}

\subsection{Nourdin-Peccati Bound}

Let $\Lipnorm{\cdot}$ and $\BLnorm{\cdot}$ be the Lipschitz and bounded Lipschitz seminorms,\footnote{For $h:\RR^d\to\RR^k$, $\Lipnorm{h}:=\sup_{x\neq y}\frac{\|h(x)-h(y)\|_{\RR^k}}{\|x-y\|_{\RR^d}}$. For $h:\RR^d\to\RR$, $\BLnorm{h}:=\Lipnorm{h}+\normsup{h}$.} respectively. Consider the following separating classes:
$$\FMclass(\RR^d):=\left\{h:\RR^d\to\RR\ \mid\ \BLnorm{h}\leq1\right\},\footnote{This class implies the so-called Fourtet-Mourier distance.}\hspace{.7cm} \Wclass(\RR^d):=\left\{h:\RR^d\to\RR\ \mid\ \Lipnorm{h}\leq1\right\},\footnote{This class implies the so-called Wasserstein distance.}$$
$${\rm and}\quad\quad \PZclass(\RR^d):=\left\{h:\RR^d\to\RR\ \mid\ h\in \mathcal{C}^2(\RR^d), \Lipnorm{h}\leq1,\Lipnorm{\nabla h}\leq1\right\}.\footnote{The letters PZ stand for the names Peccati-Zheng, the authors of the paper where this class was used to define a distance.}$$
If $\Hclass$ represents a separating class, Stein's method tells us that
$$d_{\Hclass}(Z,F):=\sup_{h\in\Hclass}\abs{\mean{h(Z)}-\mean{h(F)}}\leq \sup_{h\in\Hclass}\abs{\mean{\HSinnprod{\Sigma}{\Hess_{f_h}(F)}-\Rinnprod{F}{\nabla f_h(F)}{d}}},$$
where $f_h$ is a solution of the so-called Stein's equation. Furthermore, in \cite[Lemma 3]{Chatterjee} the authors showed that
\begin{align}\label{cotas}
\normsup{\HSnorm{\Hess_{f_h}}}\leq \boldsymbol{k}_0\Lipnorm{h}\quad \&\quad \Lipnorm{\Hess_{f_h}}\leq \boldsymbol{k}_1\Lipnorm{\nabla h}.
\end{align}

We can now state the main theorem of this paper. 
\begin{theorem}\label{main}(NP Bound)\\
Suppose that $Z\sim\mathcal{N}(m,\Sigma)$ with a positive definite covariance matrix $\Sigma$. Let $F=(F^{(1)},\dots,F^{(d)})$ be such that $\mean{F}=m$ and $F^{(i)}\in$ \domDWJ, for all $i$. Then,
\begin{align}\label{mainineq1}
d_{\Hclass}(Z,F)\leq \boldsymbol{k}\left(\mean{\HSnorm{\Sigma-\gF(F)}}+\mean{\Hprod{\abs{x}\left(\sum_{i=1}^d\abs{\MDer F^{(i)}}\right)^2}{\sum_{i=1}^d\abs{\MDer L^{-1}F^{(i)}}}}\right),
\end{align}
where $\gF^{i,j}(x):=\mean{\left.\innprod{\MDer F^{(i)}}{-\MDer L^{-1}F^{(j)}}\right|F=x}$, and $\hilbert$ is the underlying Hilbert space. If $\nu\equiv0$ (no jump part, i.e., $x=0$), then $\Hclass$ could be any of the three separating classes above, but if the jump part is present then $\Hclass=\PZclass$ and $\hilbert=\Lspa^2_\mu$.
\end{theorem}
\begin{proof}
The case where $\sigma\neq0$ and $\nu\equiv0$ was shown already in \cite{Nourdin0} and \cite{Nourdinmult}. The case where $\sigma=0$ and $\nu\not\equiv0$ was discussed in \cite{Peccati} and \cite{Peccati2}. Now, consider the case $\sigma\neq0$ and $\nu\not\equiv0$. By Proposition \ref{NPintbypartsWP} (with $k=2$ in (\ref{chainrulemultvar2f})) we get
\begin{align*}
d_{\PZclass}(Z,F)&\leq\sup_{h\in\PZclass}\abs{\mean{\HSinnprod{\Sigma}{\Hess_{f_h}(F)}}-\mean{\Rinnprod{F}{\nabla f_h(F)}{d}}}\\
&\leq\sup_{h\in\PZclass}\mean{\abs{\HSinnprod{\Sigma-\gF(F)}{\Hess_{f_h}(F)}}}\\
&\quad\quad\quad+\sup_{h\in\PZclass}\sum_{i=1}^d\mean{\Lqprod{\sum_{\abs{\beta}=2}\abs{\frac{\partial^\beta\partial_if_h(F+\theta_zx\MDer_zF)x(\MDer_zF)^\beta}{\beta!}}}{\abs{-\MDer L^{-1}F^{(i)}}}{\mu}}\\
&\leq\sup_{h\in\PZclass}\mean{\HSnorm{\Sigma-\gF(F)}\HSnorm{\Hess_{f_h}(F)}}\\
&\quad\quad\quad+\sup_{h\in\PZclass}\sup_{i,\beta}\frac{\normsup{\partial^\beta\partial_if_h}}{\beta!}\mean{\Lqprod{\abs{x}\sum_{\abs{\beta}=2}\abs{\MDer_zF}^\beta}{\sum_{i=1}^d\abs{\MDer L^{-1}F^{(i)}}}{\mu}}\\
&\leq \sup_{h\in\PZclass}\normsup{\HSnorm{\Hess_{f_h}}}\cdot\mean{\HSnorm{\Sigma-\gF(F)}}\\
&\quad\quad\quad+ \boldsymbol{C}\sup_{h\in\PZclass}\Lipnorm{\Hess_{f_h}}\cdot\mean{\Lprod{\abs{x}\left(\sum_{i=1}^d\abs{\MDer F^{(i)}}\right)^2}{\sum_{i=1}^d\abs{\MDer L^{-1}F^{(i)}}}{\mu}},
\end{align*}
where the bounds \eqref{cotas} along with $\boldsymbol{k}_0+\boldsymbol{C}\boldsymbol{k}_1\leq\boldsymbol{k}$ justify the NP bound. In the last step we used the fact that $\sum_{\abs{\beta}=2}\abs{\MDer_zF}^\beta=\sum_{i,j=1}^d\abs{\MDer_zF^{(i)}}\abs{\MDer_zF^{(j)}}=\left(\sum_{i=1}^d\abs{\MDer F^{(i)}}\right)^2$.
\end{proof}

\bigskip

In terms of computations, the inequality \eqref{mainineq1} is not as tractable as we would like it to be. The following corollary can deal with this issue.

\begin{corollary}\label{mainqchaos}(Second Order Poincar\'e Inequality)\\
Let $Z\sim\mathcal{N}(m,\Sigma)$ with a positive definite covariance matrix $\Sigma$. Let $F=(F^{(1)},\dots,F^{(d)})$ be such that $\mean{F}=m$, Var$[F]=\Sigma$ and $F^{(i)},\MDer F^{(i)}\in$ \domDWJ, for all $i$. Then,
\begin{align}
\hspace{-1cm} &d_{\Hclass}(Z,F)\leq \boldsymbol{C}\hspace{-.1cm}\left(\sum_{i,j=1}^d\left(\hspace{-.1cm}\mean{\opnorm{\boldsymbol{D}^2F^{(i)}}^4}^{\frac{1}{4}}\hspace{-.1cm}\mean{\Hnorm{\boldsymbol{D}F^{(j)}}^4}^{\frac{1}{4}}\hspace{-.2cm}+\mean{\Hnorm{\Hprod{x}{\boldsymbol{D}^2F^{(i)}\boldsymbol{D}^2F^{(j)}}}^2}^{\frac{1}{2}}\right)\hspace{-.1cm}+\sum_{i=1}^d\mean{\Hprod{\abs{x}}{\abs{\boldsymbol{D}F^{(i)}}^3}}\right)\label{ineqqchaos1}\\
\hspace{-1cm}&\leq \boldsymbol{C}\hspace{-.1cm}\left(\sum_{i,j=1}^d\left(\hspace{-.1cm}\mean{\HHnorm{\boldsymbol{D}^2F^{(i)}\otimes_1\boldsymbol{D}^2F^{(i)}}^2}^{\frac{1}{4}}\hspace{-.1cm}\mean{\Hnorm{\boldsymbol{D}F^{(j)}}^4}^{\frac{1}{4}}\hspace{-.2cm}+\mean{\Hnorm{\Hprod{x}{\boldsymbol{D}^2F^{(i)}\boldsymbol{D}^2F^{(j)}}}^2}^{\frac{1}{2}}\right)\hspace{-.1cm}+\sum_{i=1}^d\mean{\Hprod{\abs{x}}{\abs{\boldsymbol{D}F^{(i)}}^3}}\right),\label{ineqqchaos2}
\end{align}
where $\hilbert$ is the underlying Hilbert space, and $\opnorm{\cdot}$ is the operator norm.\footnote{Consider the operator $T:\hilbert\to\hilbert$ such that $T(h)=\Hprod{h}{\boldsymbol{D}^2F}$. The operator norm of $T$ is what we have denoted by $\opnorm{\boldsymbol{D}^2F}$.} If $\nu\equiv0$ (no jump part, i.e., $x=0$), then $\Hclass$ could be any of the three separating classes above, but if the jump part is present then $\Hclass=\PZclass$, $\hilbert=\Lspa^2_\mu$ and $F^{(i)}$ must lie in a fixed Wiener-Poisson chaos, that is, $F^{(i)}=I_{q_i}(f_i)$ for some $q_i\in\NN$, for all $i$.
\end{corollary}

\pagebreak

\begin{proof}
The one dimensional version (with $\nu\equiv0$) of this inequality was proved in \cite{Nourdin}, but no multidimensional inequality has been worked out. Since $d_{\Hclass}(Z,F)=d_{\Hclass}(Z-m,F-m)$, then we can assume without loss of generality that $m=0$. The last term is not zero only if $\nu\not\equiv0$, and in this case $-\MDer L^{-1}F^{(i)}=\frac{1}{q_i}\MDer F^{(i)}$, with $1\leq q_i$. Hence, applying $\left(\sum_{i=1}^d\abs{a_i}\right)^3\leq d^3\sum_{i=1}^d\abs{a_i}^3$,
$$\mean{\Hprod{\abs{x}\left(\sum_{i=1}^d\abs{\MDer F^{(i)}}\right)^2}{\sum_{i=1}^d\abs{\MDer L^{-1}F^{(i)}}}}\leq\mean{\Hprod{\abs{x}}{\left(\sum_{i=1}^d\abs{\boldsymbol{D}F^{(i)}}\right)^3}}\leq d^3\sum_{i=1}^d\mean{\Hprod{\abs{x}}{\abs{\boldsymbol{D}F^{(i)}}^3}}.$$

\bigskip

To show the other terms, consider the random matrix $G^{i,j}_F:=\innprod{\MDer F^{(i)}}{-\MDer L^{-1}F^{(j)}}$. Note that $\mean{\gF(F)}=\mean{G_F}=\bigl(\mean{F^{(i)}F^{(j)}})_{1\leq i, j\leq d}=\Sigma$, and by Jensen's inequality (for conditional expectation) along with H$\ddot{\text{o}}$lder's we have that
$$\mean{\HSnorm{\Sigma-\gF(F)}}\leq\mean{\HSnorm{\Sigma-\gF(F)}^2}^{\frac{1}{2}}=\sqrt{\sum_{i,j=1}^d\Var{\gF^{i,j}(F)}}\leq\sum_{i,j=1}^d\sqrt{\Var{G^{i,j}_F}}.$$
By the Gaussian Poincar\'e inequality (Proposition \ref{poincareineq}) we have that $\sqrt{\Var{G^{i,j}_F}}\leq \sqrt{\mean{\Hnorm{\boldsymbol{D}G^{i,j}_F}^2}}$. Also, by the product rule \eqref{productrule2}, the triangular inequality, and H$\ddot{\text{o}}$lder we get 
\begin{align*}
\hspace{-1cm}&\mean{\Hnorm{\boldsymbol{D}G^{i,j}_F}^2}^{\frac{1}{2}}\hspace{-.2cm}\leq\mean{\Hnorm{\Hprod{\boldsymbol{D}^2F^{(i)}}{-\boldsymbol{D}L^{-1}F^{(j)}}}^2}^{\frac{1}{2}}\hspace{-.2cm}+\mean{\Hnorm{\Hprod{\boldsymbol{D}F^{(i)}}{-\boldsymbol{D}^2L^{-1}F^{(j)}}}^2}^{\frac{1}{2}}\hspace{-.2cm}+\mean{\Hnorm{\Hprod{x\boldsymbol{D}^2F^{(i)}}{-\boldsymbol{D}^2L^{-1}F^{(j)}}}^2}^{\frac{1}{2}}\\
\hspace{-1cm}&\quad\quad\leq \mean{\opnorm{\boldsymbol{D}^2F^{(i)}}^4}^{\frac{1}{4}}\mean{\Hnorm{\boldsymbol{D}L^{-1}F^{(j)}}^4}^{\frac{1}{4}}+\mean{\Hnorm{\boldsymbol{D}F^{(i)}}^4}^{\frac{1}{4}}\mean{\opnorm{\boldsymbol{D}^2L^{-1}F^{(j)}}^4}^{\frac{1}{4}}+\mean{\Hnorm{\Hprod{x\boldsymbol{D}^2F^{(i)}}{-\boldsymbol{D}^2L^{-1}F^{(j)}}}^2}^{\frac{1}{2}},
\end{align*}
because $\Hnorm{\Hprod{\boldsymbol{D}^2F}{\boldsymbol{D}G}}^2\leq \opnorm{\boldsymbol{D}^2F}^2\Hnorm{\boldsymbol{D}G}^2$.

\bigskip

{\bf $\boldsymbol{\bullet\ }$ If $\boldsymbol{\nu\equiv0}$:}\\
In this scenario $x=0$, hence $\mean{\Hnorm{\Hprod{x\boldsymbol{D}^2F^{(i)}}{-\boldsymbol{D}^2L^{-1}F^{(j)}}}^2}^{\frac{1}{2}}=0$. In \cite[Proposition 3.1]{Nourdin} there is a proof that $\mean{\Hnorm{\boldsymbol{D}L^{-1}F}^4}^{\frac{1}{4}}\leq \mean{\Hnorm{\boldsymbol{D}F}^4}^{\frac{1}{4}}$ and $\mean{\opnorm{\boldsymbol{D}^2L^{-1}F^{(j)}}^4}^{\frac{1}{4}}\leq \frac{1}{2}\mean{\opnorm{\boldsymbol{D}^2F^{(j)}}^4}^{\frac{1}{4}}$.

\bigskip

{\bf $\boldsymbol{\bullet\ }$ If $\boldsymbol{\nu\not\equiv0}$:}\\
In this case $F^{(j)}$ is restricted to a fixed Wiener-Poisson chaos, i.e., $F^{(j)}=I_{q_j}(f_j)$ and $-\MDer L^{-1}F^{(j)}=\frac{1}{q_j}\MDer F^{(j)}$. Using this we get $\mean{\Hnorm{\boldsymbol{D}L^{-1}F}^4}^{\frac{1}{4}}=\frac{1}{q_j}\mean{\Hnorm{\boldsymbol{D}F}^4}^{\frac{1}{4}}$, $\mean{\opnorm{\boldsymbol{D}^2L^{-1}F^{(j)}}^4}^{\frac{1}{4}}=\frac{1}{q_j}\mean{\opnorm{\boldsymbol{D}^2F^{(j)}}^4}^{\frac{1}{4}}$, and $\mean{\Hnorm{\Hprod{x\boldsymbol{D}^2F^{(i)}}{-\boldsymbol{D}^2L^{-1}F^{(j)}}}^2}^{\frac{1}{2}}=\frac{1}{q_j}\mean{\Hnorm{\Hprod{x}{\boldsymbol{D}^2F^{(i)}\boldsymbol{D}^2F^{(j)}}}^2}^{\frac{1}{2}}$.

\bigskip

\noindent In both cases there is a constant $\boldsymbol{C}$ such that
$$\sum_{i,j=1}^d\mean{\Hnorm{\boldsymbol{D}G^{i,j}_F}^2}^{\frac{1}{2}}\hspace{-.2cm}\leq \boldsymbol{C}\left(\sum_{i,j=1}^d\mean{\opnorm{\boldsymbol{D}^2F^{(i)}}^4}^{\frac{1}{4}}\mean{\Hnorm{\boldsymbol{D}F^{(j)}}^4}^{\frac{1}{4}}+\sum_{i,j=1}^d\mean{\Hnorm{\Hprod{x}{\boldsymbol{D}^2F^{(i)}\boldsymbol{D}^2F^{(j)}}}^2}^{\frac{1}{2}}\right),$$
proving \eqref{ineqqchaos1}. In order to show \eqref{ineqqchaos2}, just invoke the proof of Proposition 4.1 in \cite{Nourdin}, which remains valid in this framework, to get $\opnorm{\boldsymbol{D}^2F^{(i)}}^4\leq \HHnorm{\boldsymbol{D}^2F^{(i)}\otimes_1\boldsymbol{D}^2F^{(i)}}^2$.
\end{proof}

\pagebreak

\begin{preremark}\label{conditions}
This corollary basically says that if we want to show gaussian convergence for a family of random vectors $F_T=(F^{(1)}_T,\dots,F^{(d)}_T)$ (living in a fixed chaos if $\nu\not\equiv0$) it is sufficient to check the following conditions {\bf for all $\boldsymbol{i}$}:
\begin{enumerate}

\item {\bf Expectation of the First Derivative Norm:}
\begin{align}\label{2cond1}
\mean{\Hnorm{\boldsymbol{D}F^{(i)}_T}^4} = O(1)\ \text{ as }\ T\to\infty.
\end{align}

\item {\bf Expectation of the First Derivative Cube:}
\begin{align}\label{2cond2}
\mean{\Hprod{\abs{x}}{\abs{\boldsymbol{D}F^{(i)}_T}^3}}\rightarrow0\ \text{ as } \ T\rightarrow\infty.
\end{align}

\item {\bf Expectation of the Contraction Norm:}
\begin{align}\label{2cond3}
\mean{\HHnorm{\boldsymbol{D}^2F^{(i)}_T\otimes_1\boldsymbol{D}^2F^{(i)}_T}^2}\rightarrow0\ \text{ as } \ T\rightarrow\infty.
\end{align}

\item {\bf Expectation of the Second Derivative Norm:}
\begin{align}\label{2cond4}
\mean{\Hnorm{\Hprod{x}{\boldsymbol{D}^2F^{(i)}_T\boldsymbol{D}^2F^{(j)}_T}}}\rightarrow0\ \text{ as } \ T\rightarrow\infty.
\end{align}

\item {\bf Existence of the Variance}
\begin{align}\label{2cond5}
\Var{F_T}\rightarrow \Sigma\ \text{ exists as } T\to\infty.
\end{align}

\end{enumerate}
Due to the Gaussian Poincar\'e inequality, $\Var{F^{(i)}_T}\leq \mean{\Hnorm{\boldsymbol{D}F^{(i)}_T}^2}\leq\sqrt{\mean{\Hnorm{\boldsymbol{D}F_T}^4}}$, so the variance of $F^{(i)}_T$ will go to 0 if the expectation of the first Malliavin derivative norm goes to 0. This is why condition (\ref{2cond1}) is necessary, and the convergence to zero of the distance relies on conditions (\ref{2cond2}), (\ref{2cond3}) and (\ref{2cond4}).
\end{preremark}

\section{Special Cases and Applications}

\subsection{The Wiener Space Case:\\Linear functionals of Gaussian-subordinated fields}

When we are working in this space the jump size is always zero, that is $\nu\equiv0$, so the upper bound for the Wasserstein distance becomes
\begin{align}\label{wienercase}
d_{\Wclass}(Z,F)\leq \boldsymbol{k}\mean{\HSnorm{\Sigma-\gF(F)}},
\end{align}
which coincides perfectly with the bounds computed in \cite{Nourdin0} and \cite{Nourdinmult}.

\bigskip

Since Corollary \ref{mainqchaos}, in this space, is true for all $F\in\mathbb{D}^{2,4}$ and not just for functionals in a fixed Wiener chaos, the authors of \cite{Nourdin} proved a very useful central limit theorem for linear functionals of Gaussian-subordinated fields. Before stating it, some notation is introduced: Let $X_t$ be a centered Gaussian stationary process and define $C(t)=\mean{X_0X_t}=\mean{X_sX_{t+s}}$, its covariance function. By Remark \ref{hilbertderivative}, we know that the Malliavin derivative of $X_t$ is well defined. Let $T>0$, $Z\sim \mathcal{N}\bigl(0,C(0)\bigr)$ and $\ f:\RR\rightarrow \RR$ be a real function of class $\mathcal{C}^2$ not constant such that $\mean{|f(Z)|}<\infty$ and $\mean{|f''(Z)|^4}<\infty$. In order to simplify the notation, the following random sequence is defined,
$$F_T := T^{-\frac{1}{2}}\int_0^T\bigl(f(X_t)-\mean{f(Z)}\bigr)\,dt.$$
Their result is stated as follows,
\begin{lemma}\label{finitecase}
Suppose that $\int_\RR \abs{C(t)}\,dt<\infty$, and assume that $f$ is a {\bf symmetric} real valued function. Then $\lim_{T\rightarrow\infty}\Var{F_T}:=\Sigma^2\in(0,\infty)$ exists and as $T\rightarrow\infty$
$$F_T\stackrel{\text{law}}{\longrightarrow}N\sim\mathcal{N}(0,\Sigma^2).$$
\end{lemma}

Our goal in this subsection is to extend this result to the case when $\int_\RR \abs{C(t)}\,dt=\infty$. This is achievable under some conditions on the decay rate of the covariance. In fact, it is very handy that for this functional the conditions (\ref{2cond1})-(\ref{2cond5}) reduce to just one condition on the covariance of the underlying stationary Gaussian process $X_t$. Let $V(T)$ be a strictly positive continuous function with $V(T)\rightarrow0$ as $T\to\infty$ such that either $TV(T)\rightarrow0$ or $V\in\mathcal{C}^1$ and $TV'(T)\rightarrow0$ as $T\rightarrow\infty$. The following is the condition on the covariance that replaces the five conditions in Remark \ref{conditions}.\\

\noindent{\bf Condition $\boldsymbol{\ast}$:}
Either $\int_\RR \abs{C(t)}\,dt<\infty$ or $V(T)$ (with the above characteristics) exists such that,
$$\frac{C(T)}{V(T)}\xrightarrow[T\rightarrow \infty]{}M\neq0.$$
$V(T)$ represents the decay rate for the covariance function.
Consider the following function
\begin{align*}
\widetilde{V}(T)=\begin{cases}
T & \text{if } \ \int_0^\infty\abs{C(x)}\,dx<\infty\\
\int_0^T\int_0^yV(x)\,dx\,dy & \text{if } \ \int_0^\infty\abs{C(x)}\,dx=\infty
\end{cases}.
\end{align*}
Let $\mathcal{M}_C:=\{f\in\mathcal{C}^2\ \mid f \text{ is symmetric if } \int_\RR \abs{C(t)}\,dt<\infty \text{ or } \mean{f(Z)Z}\neq0 \text{ if } \int_\RR \abs{C(t)}\,dt=\infty\}$ and rewrite the functional $F_T$ as follows,
$$F_T := \widetilde{V}(T)^{-\frac{1}{2}}\int_0^T\bigl(f(X_t)-\mean{f(Z)}\bigr)\,dt.$$

\begin{theorem}\label{main2}
Suppose that condition $\ast$ is verified by $C(t)$ and that $f\in\mathcal{M}_C$. Then $\lim_{T\rightarrow\infty}\text{Var}[F_T]:=\Sigma^2\in(0,\infty)$ exists and as $T\rightarrow\infty$
$$F_T\stackrel{\text{law}}{\longrightarrow}N\sim\mathcal{N}(0,\Sigma^2).$$
Furthermore, if $\int_\RR \abs{C(t)}\,dt=\infty$, then $\Sigma^2=2M\bigl(\mean{f(Z)Z}\bigr)^2$.
\end{theorem}
Before tackling this theorem, we need to verify some facts that will simplify the proof.
\begin{prop}\label{computations}
Suppose that $\int_\RR \abs{C(t)}\,dt=\infty$. Then as $T\rightarrow\infty$,
\begin{enumerate}

\item $\bigl(\int_0^TV(x)\,dx\bigr)^{-1}\int_0^T\abs{C(t)}\,dt=O(1)$.

\item $\widetilde{V}(T)^{-1}\int_{[0,T]^2}\abs{C(t-s)}\,ds\,dt=O(1)$.

\item
\begin{itemize}

\item If $TV(T)\rightarrow0$:\\
$\widetilde{V}(T)^{-2}T\bigl(\int_0^T\abs{C(t)}\,dt\bigr)^3=O\bigl(\max\{V(T),TV(T)^2\bigl(\int_0^TV(x)\,dx\bigr)^{-1}\}\bigr)$.

\item If $TV(T)\nrightarrow0$ and $TV'(T)\rightarrow0$\\
$\widetilde{V}(T)^{-2}T\bigl(\int_0^T\abs{C(t)}\,dt\bigr)^3=O\bigl(\max\{V(T),TV'(T)\}\bigr)$.

\end{itemize}

\item For fixed $q\geq1$:
\begin{align*}
\widetilde{V}(T)^{-1}\int_{[0,T]^2}C(t-s)^q\,ds\,dt\rightarrow 2M1_{\{q=1\}}=
\begin{cases}
2M & \text{if } \ q=1\\
0 & \text{if } \ q\neq 1
\end{cases}.
\end{align*}

\end{enumerate}
\end{prop}

\begin{proof}
The proof just involves simple applications of L'H$\hat{\text{o}}$pital's rule (L).
\begin{enumerate}

\item $$\lim_{T\rightarrow\infty}\frac{\int_0^T\abs{C(t)}\,dt}{\int_0^TV(x)\,dx}\stackrel{\text{L}}{=}\lim_{T\rightarrow\infty}\frac{\abs{C(T)}}{V(T)}=\abs{M}.$$

\item Notice first that $\int_{[0,T]^2}\abs{C(t-s)}\,ds\,dt=2\int_0^T\int_0^t\abs{C(x)}\,dx\,dt$ so
$$\lim_{T\rightarrow\infty}\frac{\int_{[0,T]^2}\abs{C(t-s)}\,ds\,dt}{\widetilde{V}(T)}=\lim_{T\rightarrow\infty}\frac{2\int_0^T\int_0^t\abs{C(x)}\,dx\,dt}{\int_0^T\int_0^yV(x)\,dx\,dy}\stackrel{\text{L}}{=}2\lim_{T\rightarrow\infty}\frac{\abs{C(T)}}{V(T)}=2\abs{M}.$$

\item $\lim_{T\rightarrow\infty}\frac{T\bigl(\int_0^T\abs{C(t)}\,dt\bigr)^3}{\widetilde{V}(T)^2}=\lim_{T\rightarrow\infty}\biggl(\underbrace{\frac{\int_0^T\abs{C(t)}\,dt}{\int_0^TV(x)\,dx}}_{=O(1)}\biggr)^3\lim_{T\rightarrow\infty}\frac{T\bigl(\int_0^TV(x)\,dx\bigr)^3}{\widetilde{V}(T)^2}$.
\begin{itemize}

\item If $TV(T)\rightarrow0$:
\begin{align*}
\lim_{T\rightarrow\infty}\frac{T\bigl(\int_0^TV(x)\,dx\bigr)^3}{\widetilde{V}(T)^2}&=\lim_{T\to\infty}\frac{\bigl(\int_0^TV(x)\,dx\bigr)^2}{\widetilde{V}(T)}\lim_{T\rightarrow\infty}\frac{T\int_0^TV(x)\,dx}{\widetilde{V}(T)}\\
&\stackrel{\text{L}}{=}\lim_{T\rightarrow\infty}\bigl(2V(T)\bigr)\lim_{T\rightarrow\infty}\biggl(1+\frac{TV(T)}{\int_0^TV(x)\,dx}\biggr)=2\lim_{T\rightarrow\infty}\biggl(V(T)+\frac{TV(T)^2}{\int_0^TV(x)\,dx}\biggr).
\end{align*}

\item If $TV(T)\nrightarrow0$ and $TV'(T)\rightarrow0$\\
\begin{align*}
\lim_{T\rightarrow\infty}\frac{T\bigl(\int_0^TV(x)\,dx\bigr)^3}{\widetilde{V}(T)^2}&\stackrel{\text{L}}{=}\lim_{T\rightarrow\infty}\frac{\bigl(\int_0^TV(x)\,dx\bigr)^2+3TV(T)\bigl(\int_0^TV(x)\,dx\bigr)}{2\widetilde{V}(T)}\\
&\stackrel{\text{L}}{=}\biggl(\lim_{T\to\infty}\frac{5V(T)+3TV'(T)}{2}+\lim_{T\rightarrow\infty}\frac{3TV(T)^2}{2\int_0^TV(x)\,dx}\biggr)\stackrel{\text{L}}{=}\lim_{T\rightarrow\infty}\biggl(4V(T)+\frac{9}{2}TV'(T)\biggr).
\end{align*}

\end{itemize}

\item If for $q>1$, either $\lim_{T\rightarrow\infty}\int_{[0,T]^2}C(t-s)^q\,ds\,dt<\infty \ \text{ or }\ \lim_{T\rightarrow\infty}\int_0^TC(x)^q\,dx<\infty$, then the result will follow trivially. So let's assume that both go to infinity as $T$ goes to infinity.
$$ \lim_{T\rightarrow\infty}\frac{\int_{[0,T]^2}C(t-s)^q\,ds\,dt}{\widetilde{V}(T)}=\lim_{T\rightarrow\infty}\frac{2\int_0^T\int_0^tC(x)^q\,dx\,dt}{\int_0^T\int_0^yV(x)\,dx\,dy}\stackrel{\text{L}}{=}2\lim_{T\rightarrow\infty}\frac{C(T)^q}{V(T)}=2\lim_{T\rightarrow\infty}\underbrace{\biggl(\frac{C(T)}{V(T)}\biggr)^q}
_{\rightarrow M^q}\underbrace{V(T)^{(q-1)}}_{\rightarrow 0 \text{ if } q>1}=2M1_{\{q=1\}}.$$

\end{enumerate}
\end{proof}
Proof of Theorem \ref{main2}.
\begin{proof}
Notice that if $\int_\RR\abs{C(t)}\,dt<\infty$ then Theorem \ref{main2} reduces to Lemma \ref{finitecase} and there is nothing left to prove. Assume then, that $\int_\RR\abs{C(t)}\,dt=\infty$. Due to Remark \ref{conditions}, it is enough to check that condition $\ast$ implies conditions (\ref{2cond1}), (\ref{2cond3}) and (\ref{2cond5}), because \eqref{2cond2} and \eqref{2cond4} trivially holds.
\begin{itemize}

\item {\bf Expectation of the First Derivative Norm:}\\
\begin{itemize}

\item First Malliavin Derivative:
$$\boldsymbol{D} F_T = \widetilde{V}(T)^{-\frac{1}{2}}\int_0^Tf'(X_t)1_{[0,t]}\,dt.$$

\item Norm of the First Malliavin Derivative:
$$\Hnorm{\boldsymbol{D} F_T}^2 = \widetilde{V}(T)^{-1}\int_{[0,T]^2}\hspace{-.7cm}f'(X_t)f'(X_s)\Hprod{1_{[0,t]}}{1_{[0,s]}} dtds = \widetilde{V}(T)^{-1}\int_{[0,T]^2}\hspace{-.7cm}f'(X_t)f'(X_s)C(t-s) \,dt\,ds.$$
Then,
$$\Hnorm{\boldsymbol{D} F_T}^4 = \widetilde{V}(T)^{-2}\int_{[0,T]^4}\hspace{-.7cm}f'(X_t)f'(X_s)f'(X_u)f'(X_v)C(t-s)C(u-v) \,dt\,ds\,du\,dv.$$

\item Expectation of the First Malliavin Derivative Norm:\\
By using H$\ddot{\text{o}}$lder (twice) on the expectation and by the stationarity of $X_t$ we have the bound,
$$\abs{\mean{f'(X_t)f'(X_s)f'(X_u)f'(X_v)}}\leq \mean{\abs{f'(Z)}^4},$$
finally recovering the power we get,
$$\mean{\Hnorm{\boldsymbol{D} F_T}^4}\leq \mean{|f'(Z)|^4}\underbrace{\biggl(\widetilde{V}(T)^{-1}\int_{[0,T]^2}\abs{C(t-s)} dtds\biggr)^2}_{=O(1) \text{ by Proposition \ref{computations}}}.$$

\end{itemize}
All this proves that,
\begin{align*}
\boxed{\mean{\Hnorm{\boldsymbol{D} F_T}^4}^{\frac{1}{4}}=O(1)\ \text{ as }\ T\to\infty.}
\end{align*}

\item {\bf Expectation of the Contraction Norm:}\\
In the same way we get,

\begin{itemize}

\item Second Malliavin Derivative:
$$\boldsymbol{D}^2 F_T = \widetilde{V}(T)^{-\frac{1}{2}}\int_0^Tf''(X_t)1_{[0,t]}^{\otimes2}\,dt.$$

\item Contraction of Order 1:
\begin{align*}
\boldsymbol{D}^2F_T\otimes_1\boldsymbol{D}^2F_T =& \widetilde{V}(T)^{-1}\int_{[0,T]^2}\hspace{-.7cm}f''(X_t)f''(X_s)1_{[0,t]}\otimes 1_{[0,s]}\left\langle 1_{[0,t]},1_{[0,s]}\right\rangle_{\hilbert} \,dt\,ds\\
=& \widetilde{V}(T)^{-1}\int_{[0,T]^2}\hspace{-.7cm}f''(X_t)f''(X_s)1_{[0,t]}\otimes 1_{[0,s]}C(t-s) \,dt\,ds.
\end{align*}

\item Norm of the Contraction:
\begin{align*}
\HHnorm{\boldsymbol{D}^2F_T\otimes_1\boldsymbol{D}^2F_T}^2 &= \widetilde{V}(T)^{-2}\int_{[0,T]^4}\hspace{-.7cm}f''(X_t)f''(X_s)f''(X_u)f''(X_v)C(t-s)C(u-v)\\
&\hspace{5cm}\times\left\langle 1_{[0,t]},1_{[0,u]}\right\rangle_{\hilbert}\left\langle 1_{[0,s]},1_{[0,v]}\right\rangle_{\hilbert} \,dt\,ds\,du\,dv\\
&= \widetilde{V}(T)^{-2}\int_{[0,T]^4}\hspace{-.7cm}f''(X_t)f''(X_s)f''(X_u)f''(X_v)C(t-s)C(u-v)C(t-u)C(s-v) \,dt\,ds\,du\,dv.
\end{align*}

\item Expectation of the Contraction Norm:\\
By using H$\ddot{\text{o}}$lder in the same way as above we get,
$$\mean{\HHnorm{\boldsymbol{D}^2F_T\otimes_1\boldsymbol{D}^2F_T}^2} \leq \mean{|f''(Z)|^4}\widetilde{V}(T)^{-2}\int_{[0,T]^4}\abs{C(t-s)C(u-v)C(t-u)C(s-v)} \,dt\,ds\,du\,dv.$$
Now, let's make the change of variable $y=(t-s,u-v,t-u,v)$, and let's denote the new region by $\widetilde{\Omega}\times[0,T]$. So,
\begin{align*}
\mean{\HHnorm{\boldsymbol{D}^2F_T\otimes_1\boldsymbol{D}^2F_T}^2} &\leq \mean{|f''(Z)|^4}\widetilde{V}(T)^{-2}\int_0^T\int_{\widetilde{\Omega}}\abs{C(y_1)C(y_2)C(y_3)C(y_2+y_3-y_1)} \,dy_1\,dy_2\,dy_3\,dv\\
&= \mean{|f''(Z)|^4}\widetilde{V}(T)^{-2}T\int_{\widetilde{\Omega}}\abs{C(y_1)C(y_2)C(y_3)C(y_2+y_3-y_1)} \,dy_1\,dy_2\,dy_3.
\end{align*}
We take into account that by Cauchy-Schwarz, for all $t\in\RR$,
$$\abs{C(t)}=\overbrace{\frac{\abs{\mean{X_0X_t}}}{\sqrt{\text{Var}[X_0]\text{Var}[X_t]}}}^{\leq 1}\overbrace{\sqrt{\text{Var}[X_0]\text{Var}[X_t]}}^{=C(0)}\leq C(0).$$
As it is clear that $\widetilde{\Omega}\subset[-T,T]^3$ and since the integrand is a non-negative even function, we can deduce that,
\begin{align*}
\mean{\HHnorm{\boldsymbol{D}^2F_T\otimes_1\boldsymbol{D}^2F_T}^2} &\leq \mean{|f''(Z)|^4}C(0)\widetilde{V}(T)^{-2}T\biggl(2\int_{[0,T]}\abs{C(y)}\,dy\biggr)^3\\
&= 8\mean{|f''(Z)|^4}C(0)\underbrace{\widetilde{V}(T)^{-2}T\biggl(\int_0^T\abs{C(y)}\,dy\biggr)^3}_{\xrightarrow[T\to\infty]{}0 \text{ by Proposition \ref{computations}}}.
\end{align*}

All this proves that,
\begin{itemize}

\item If $TV(T)\rightarrow0$:
\begin{align*}
\hspace{-2cm} \boxed{\mean{\HHnorm{\boldsymbol{D}^2F_T\otimes_1\boldsymbol{D}^2F_T}^2}^{\frac{1}{4}}=O\biggl(\max\left\{V(T),TV(T)^2\biggl(\int_0^TV(x)\,dx\biggr)^{-1}\right\}^{\frac{1}{4}}\biggr)\ \text{ as }\ T\to\infty.}
\end{align*}
\\
\item If $TV(T)\nrightarrow0$ and $TV'(T)\rightarrow0$:
\begin{align*}
\boxed{\mean{\HHnorm{\boldsymbol{D}^2F_T\otimes_1\boldsymbol{D}^2F_T}^2}^{\frac{1}{4}}=O\bigl(\max\{V(T),TV'(T)\}^{\frac{1}{4}}\bigr)\ \text{ as }\ T\to\infty.}
\end{align*}

\end{itemize}
\end{itemize}

\item {\bf Existence of the Variance:}\\
Since $f\in\mathcal{M}_C$ then $\mean{f(X_0)X_0}=\mean{f(Z)Z}\neq0$. Also $H_1(x)=x$, so the first Hermite constant in the expansion (\ref{decomp}) is not 0, i.e., $c_1=\mean{f(X_0)X_0}\neq0$. Using the formula (\ref{covf}) for the covariance of $f$ we get,
\begin{align*}
\text{Var}[F_T]&=\mean{\biggl(\widetilde{V}(T)^{-\frac{1}{2}}\int_0^T\bigl(f(X_t)-\mean{f(Z)}\bigr)\,dt\biggr)^2}=\widetilde{V}(T)^{-1}\int_{[0,T]^2}\text{Cov}\bigl[f(X_t),f(X_s)\bigr]\,dt\,ds\\
&=\widetilde{V}(T)^{-1}\int_{[0,T]^2}\sum_{q=1}^\infty c_q^2q!\bigl(\mean{X_tX_s}\bigr)^q\,dt\,ds=\sum_{q=1}^\infty c_q^2q!\widetilde{V}(T)^{-1}\int_{[0,T]^2}C(t-s)^q\,dt\,ds\\
&=c_1^2\underbrace{\frac{\int_{[0,T]^2}C(t-s)\,ds\,dt}{\widetilde{V}(T)}}_{\rightarrow 2M \text{ by Proposition \ref{computations}}}+\sum_{q=2}^\infty c_q^2q!\underbrace{\frac{\int_{[0,T]^2}C(t-s)^q\,ds\,dt}{\widetilde{V}(T)}}_{\rightarrow 0,\ \forall \ q \text{ by Proposition \ref{computations}}}\xrightarrow[T\rightarrow\infty]{}2Mc_1^2.
\end{align*}
All this proves that,
\begin{align*}
\boxed{\lim_{T\rightarrow\infty}\text{Var}[F_T]=2M\bigl(\mean{f(Z)Z}\bigr)^2\in(0,\infty)\ \text{ exists}.}
\end{align*}

\end{itemize}
Since the conditions were satisfied, Theorem \ref{main2} is proved.
\end{proof}

\begin{preremark}\label{quantrate2}\rm
Notice that during the proof of this theorem we could establish an estimate for the convergence rate to normality, i.e., if $\widetilde{Z}\sim\mathcal{N}(0,1)$, then, as $T\to\infty$, we have
\begin{itemize}

\item If $TV(T)\rightarrow0$:
$$d_W\biggl(\frac{F_T}{\sqrt{\text{Var}[F_T]}},\widetilde{Z}\biggr)=O\biggl(\max\left\{V(T),TV(T)^2\biggl(\int_0^TV(x)\, dx\biggr)^{-1}\right\}^{\frac{1}{4}}\biggr).$$

\item If $TV(T)\nrightarrow0$ and $TV'(T)\rightarrow0$:
$$d_W\biggl(\frac{F_T}{\sqrt{\text{Var}[F_T]}},\widetilde{Z}\biggr)=O\bigl(\max\{V(T),TV'(T)\}^{\frac{1}{4}}\bigr).$$

\end{itemize}
In fact, it coincides with the rate obtained in \cite{Nourdin} for the case $\int_\RR \abs{C(t)}\,dt<\infty$, when $V(T)=\frac{1}{T}$ and $d_W\biggl(\frac{F_T}{\sqrt{\text{Var}[F_T]}},\widetilde{Z}\biggr)=O(T^{-\frac{1}{4}})$.
\end{preremark}

\subsubsection{Examples}
According to Theorem \ref{main2} the only condition we need to check in order to apply the central limit theorem to $F_T$ is the decay rate of the covariance function for the underlying stationary Gaussian process $X_t$ (condition $\ast$). In fact, if the decay rate is $t^{-\alpha}$ then we can apply the CLT if $\alpha\in(0,1)\cup(1,2)$, because in the case $\alpha\in(1,2)$ the integral $\int_\RR C(t)\,dt$ is finite and in the case $\alpha\in(0,1)$ the same integral is infinite but $V(T)=T^{-\alpha}\in\mathcal{C}^1$ and $TV'(T)=-\alpha T^{-\alpha}\rightarrow0$ as $T\rightarrow\infty$.

\begin{enumerate}

\item {\bf Fractional Brownian Motion (fBm):}
\begin{defn}
A fractional Brownian motion (fBm) with Hurst parameter $H\in(0,1]$ is a $\mathbb{P}$-a.s. continuous, centered Gaussian process $(B^H_t)_{t\in\RR}$ with covariance structure given by
$$\mean{B^H_tB^H_s}=\frac{1}{2}\left(\abs{t}^{2H}+\abs{s}^{2H}-\abs{t-s}^{2H}\right).$$
\end{defn}
The fractional Brownian motion enjoys the property of having stationary increments (even if they are not independent), that is, $B_{t+s}^H-B^H_s\stackrel{\text{Law}}{=}B_t^H$, for all $t,s\in\RR$ and all $H\in(0,1)$. So, if $B^H_t$ is fBm then $X_t=B^H_{t+1}-B^H_t$ is a centered Gaussian stationary process with covariance function given by
$$C_1(T)=\mean{X_TX_0}=\mean{(B^H_{T+1}-B^H_T)(B^H_1-B^H_0)}=\frac{\abs{T+1}^{2H}+\abs{T-1}^{2H}-2T^{2H}}{2}.$$
Thus, if $H>\frac{1}{2}$,
$$\lim_{T\rightarrow\infty}T^{2-2H}C_1(T)=\lim_{T\rightarrow\infty}T^2\frac{\bigl(1+\frac{1}{T}\bigr)^{2H}+\bigl(1-\frac{1}{T}\bigr)^{2H}-2}{2}=H(2H-1)=M\in(0,\infty).$$
Then, the decay rate of its covariance function is $t^{2H-2}$, i.e.,
\begin{itemize}

\item $\int_\RR\abs{C(t)}\,dt<\infty$ if $H\leq\frac{1}{2}$,

\item $\int_\RR\abs{C(t)}\,dt=\infty$ if $H>\frac{1}{2}$.

\end{itemize}
Theorem \ref{main2} is applicable to the increments of fBm for all $H\in\bigl(0,\frac{1}{2}\bigr)\cup\bigl(\frac{1}{2},1\bigr)$, and
\begin{center}
$\widetilde{V}(T)^{-\frac{1}{2}}\int_0^T\bigl(f(X_t)-\mean{f(Z)}\bigr)\,dt\stackrel{\text{law}}{\longrightarrow}N\sim\mathcal{N}(0,\Sigma^2)$ as $T\rightarrow\infty.$
\end{center}

\item {\bf Ornstein-Uhlenbeck Driven by fBm:}
\begin{defn}
The fractional-driven Ornstein-Uhlenbeck process is, in an analogous way to the classical Ornstein-Uhlenbeck, the unique\footnote{Any other stationary solution is equal to $Y^H_t$ in distribution} stationary, almost surely continuous, centered Gaussian process $Y^H_t$ that solves the Langevin stochastic differential equation 
$$dY^H_t=-\lambda Y^H_tdt+\widetilde{\sigma}dB_t^H,$$
where $\widetilde{\sigma},\lambda>0$ are constants, and $B^H_t$ is a fractional Brownian motion. This (path-wise) solution with initial condition $Y^H_0=\widetilde{\sigma}\int_{-\infty}^0e^{\lambda u}\,dB^H_u$ can be written as $Y^H_t=\widetilde{\sigma}\int_{-\infty}^te^{-\lambda(t-u)}\,dB^H_u$, where the integral is the Riemann-Stieltjes integral.
\end{defn}
Therefore, $X_t=Y^H_t-\mean{Y_0^H}$ is a centered Gaussian stationary process.
In \cite{Cheridito}, the authors proved the following lemma,
\begin{lemma}\label{lema1}
Let $H\in\bigl(0,\frac{1}{2}\bigr)\cup\bigl(\frac{1}{2},1\bigr)$ and $N\in\NN$. Then as $T\rightarrow\infty$,
$$C_2(T)=\mean{X_TX_0}=\Cov{Y^H_T}{Y^H_0}=\frac{\widetilde{\sigma}^2}{2}\sum_{n=1}^N\lambda^{-2n}\biggl(\prod_{k=0}^{2n-1}(2H-k)\biggr)T^{2H-2n}+O(T^{2H-2N-2}).$$
\end{lemma}
This basically tells us that for all $H\in\bigl(0,\frac{1}{2}\bigr)\cup\bigl(\frac{1}{2},1\bigr)$ the decay rate of $C_2(T)$ is very similar to the decay rate of $C_1(T)$ (the covariance of the fBm increments). Lemma \ref{lema1} implies that, if $H>\frac{1}{2}$,
$$\lim_{T\rightarrow\infty}T^{2-2H}C_2(T)=\frac{H(2H-1)\widetilde{\sigma}^2}{\lambda^2}=M\in(0,\infty).$$
As in Example 1, due to this rate of decrease, Theorem \ref{main2} is applicable to this process for all $H\in\bigl(0,\frac{1}{2}\bigr)\cup\bigl(\frac{1}{2},1\bigr)$, and
\begin{center}
$\widetilde{V}(T)^{-\frac{1}{2}}\int_0^T\bigl(f(X_t)-\mean{f(Z)}\bigr)\,dt\stackrel{\text{law}}{\longrightarrow}N\sim\mathcal{N}(0,\Sigma^2)$ as $T\rightarrow\infty$.
\end{center}

\end{enumerate}
According to Remark \ref{quantrate2} we can tell that $F_T:=\widetilde{V}(T)^{-\frac{1}{2}}\int_0^T\bigl(f(X_t)-\mean{f(Z)}\bigr)\,dt$ in the above examples  has a rate of convergence to normality of at least $T^{(\frac{1\vee (2H)}{4}-\frac{1}{2})}$ for all $H\in\bigl(0,\frac{1}{2}\bigr)\cup\bigl(\frac{1}{2},1\bigr)$, that is, for $\widetilde{Z}\sim\mathcal{N}(0,1)$,
$$d_W\biggl(\frac{F_T}{\sqrt{\Var{F_T}}},\widetilde{Z}\biggr)=O\bigl(T^{(\frac{1\vee (2H)}{4}-\frac{1}{2})}\bigr)\ \text{ as }\ T\to\infty.$$

\subsection{The Poisson Space Case:\\Simulation of small jumps}

Let $\{F^{(n)}_t\}$ be a sequence of stochastic processes, and $\{Z_t\}$ a stochastic process. We say that $F^{(n)}_t\stackrel{\text{law}}{\approx}Z_t$ if $\left\{F_{t_1}^{(n)},\dots,F^{(n)}_{t_d}\right\}\xrightarrow[n\to\infty]{\text{law}}\bigl\{Z_{t_1},\dots,Z_{t_d}\bigr\}$ for any set of times $\{t_1,t_2,\dots,t_d\}_{\{d\in\NN\}}$. In simulating the path of a process $\{Z_t\}$, we often need to obtain the value of $Z_t(\omega)$ for some fixed times $\{t_1,t_2,\dots,t_d\}$, i.e., we need to know the finite-dimensional distribution of $\bigl\{Z_{t_1},\dots,Z_{t_d}\bigr\}$. If $F^{(n)}_t\stackrel{\text{law}}{\approx}Z_t$, then for $n$ sufficiently large, one could use $\left\{F^{(n)}_{t_1}(\omega),\dots,F^{(n)}_{t_d}(\omega)\right\}$ in place of $\bigl\{Z_{t_1}(\omega),\dots,Z_{t_d}(\omega)\bigr\}$ for simulation.

\bigskip

In \cite{Asmussen}, the authors proved that the small jumps from a L\'evy process can be approximated by Brownian motion. Before this theorem is stated, some notation needs to be introduced: Let $Z_t$ be a L\'evy process with triplet $(b,\sigma^2,\nu)$. To isolate the small jumps, consider the variance $\sigma(\epsilon)^2=\int_{\{\abs{x}\leq \epsilon\}}x^2\,d\nu(x)$ and the small jumps process $F_t^{\epsilon}=\sigma(\epsilon)^{-1}\int\hspace{-.2cm}\int_{[0,t]\times\{\abs{x}\leq \epsilon\}}x\,d\widetilde{N}(s,x)$. Therefore, $L_t=bt+\sigma W_t+N_t^{\epsilon}+\sigma(\epsilon)F_t^{\epsilon}$ where $N_t^{\epsilon}=\sum_{s<t}\Delta X_s1_{\{\abs{\Delta X_s}>\epsilon\}}=\int\hspace{-.2cm}\int_{[0,t]\times\{\abs{x}> \epsilon\}}x\,d\widetilde{N}(s,x)$ is the part of (finitely many) jumps larger than $\epsilon$. Their theorem reads as follows:
\begin{lemma}\label{levycase}
Let $\widehat{W}_t$ be a Brownian motion independent of $W_t$. Then $F_t^{\epsilon}\stackrel{law}{\approx}\widehat{W}_t$ as $\epsilon\rightarrow0$ if $\lim_{\epsilon\rightarrow0}\frac{\sigma(\epsilon)}{\epsilon}=\infty$.
\end{lemma}
The importance of this lemma is that $L_t\stackrel{\text{law}}{\approx}b t+\sqrt{\sigma^2+\sigma(\epsilon)^2}W_t+N_t^{\epsilon}$ (for $\epsilon$ small enough), and the latter is rather easy to simulate.

\bigskip

The objective of this subsection is to extend this kind of result to functionals that are not necessarily L\'evy. To focus just on the jump part, let's assume, without loss of generality, that the triplet of the L\'evy process $L_t$ is $(0,0,\nu)$, i.e., $\nu\not\equiv0$ and $\sigma=0$. Since we are assuming that $\int_{\RR}x^2\,d\nu(x)<\infty$, then, it can be written as $L_t=I_1(1_{[0,t]})=\int\hspace{-.2cm}\int_{[0,t]\times\RR}x\,d\widetilde{N}(s,x)$. A natural generalization of a Levy process (but losing some of its properties) can be given by
$$X_t=I_1\bigl(h_t1_{[0,t]}\bigr)=\int\hspace{-.2cm}\int_{[0,t]\times\RR}h_t(s,x)x\,d\widetilde{N}(s,x).$$
Define the process $F_t^{\epsilon}=I_1\bigl(\sigma_t(\epsilon)^{-1}h_t1_{[0,t]\times\{\abs{x}\leq\epsilon\}}\bigr)=\sigma(\epsilon)^{-1}\int\hspace{-.2cm}\int_{[0,t]\times\{\abs{x}\leq \epsilon\}}h_t(s,x)x\,d\widetilde{N}(s,x)$. This means that $X_t=N_t^{\epsilon}+\sigma(\epsilon)F_t^{\epsilon}$ where $N_t^\epsilon=I_1\bigl(h_t1_{[0,t]\times\{\abs{x}>\epsilon\}}\bigr)=\int\hspace{-.2cm}\int_{[0,t]\times\{\abs{x}> \epsilon\}}h_t(s,x)x\,d\widetilde{N}(s,x)$ has finitely many jumps.

\begin{theorem}\label{smalljumpsconv}
Let $\widehat{W}$ be an isonormal Gaussian process with $\mean{\widehat{W}(f)\widehat{W}(g)}=\int_{\RR^+}f(s)g(s)\,ds$ as its covariance structure. Moreover, suppose that $h_t(s,x)=h_t(s)\in L^3(0,T)$. Then $F_t^{\epsilon}\stackrel{law}{\approx}\widehat{W}(h_t)$ as $\epsilon\rightarrow0$ if $\lim_{\epsilon\rightarrow0}\frac{\sigma(\epsilon)}{\epsilon}=\infty$.
\end{theorem}
\begin{proof}
We need to show that $F_\epsilon:=\left\{F_{t_1}^{\epsilon},\dots,F^{\epsilon}_{t_d}\right\}\xrightarrow[n\to\infty]{\text{law}}\bigl\{\widehat{W}(h_{t_1}),\dots,\widehat{W}(h_{t_d})\bigr\}=:Z$, for any set of times $\{t_1,t_2,\dots,t_d\}$ with $d\in\NN$. Note first that $F_{t_i}$ lies in the first Poisson chaos, hence conditions to use the multidimensional version of Corollary \ref{mainqchaos} are in place, and all we need to do is to make sure that the five conditions of Remark \ref{conditions} are satisfied. Moreover, $\MDer_{s,x} F_{t_i}^{\epsilon}=\sigma(\epsilon)^{-1}h_t(s)1_{[0,t]\times\{\abs{x}\leq\epsilon\}}(s,x)$, and $\MDer^2_{s,x} F_{t_i}^{\epsilon}=0$, for all $i$. This shows that \eqref{2cond3} and \eqref{2cond4} are trivially fulfilled.

\bigskip

Notice that
$$\Hnorm{\boldsymbol{D}F_{t_i}^{\epsilon}}^2=\Lqnorm{\sigma(\epsilon)^{-1}h_{t_i}1_{[0,t_i]\times\{\abs{x}\leq \epsilon\}}}{\mu}^2=\frac{\int_0^{t_i}\bigl(h_{t_i}(s)\bigr)^2\,ds\cdot\int_{\{\abs{x}\leq \epsilon\}}x^2\,d\nu(x)}{\sigma(\epsilon)^2}=\int_0^{t_i}\bigl(h_{t_i}(s)\bigr)^2\,ds.$$
Hence, $\mean{\Hnorm{\boldsymbol{D}F_{t_i}^{\epsilon}}^4}^{\frac{1}{4}}=\Lqnorm{h_{t_i}}{}=O(1)$ as $\epsilon\to0$, and condition \eqref{2cond1} holds.

\bigskip

On the other hand,
\begin{align*}
\mean{\Hprod{\abs{x}}{\abs{\boldsymbol{D}F_{t_i}^{\epsilon}}^3}}&=\frac{\dint_{[0,t_i]\times\{\abs{x}\leq\epsilon\}}\abs{xh_{t_i}(s)}^3\,d\nu(x)\,ds}{\sigma(\epsilon)^3}=\frac{\int_0^{t_i}\bigl|h_{t_i}(s)\bigr|^3\,ds\cdot\int_{\{\abs{x}\leq \epsilon\}}|x|^3\,d\nu(x)}{\sigma(\epsilon)^3}\\
&\leq\frac{\int_0^{t_i}\bigl|h_{t_i}(s)\bigr|^3\,ds\cdot\int_{\{\abs{x}\leq \epsilon\}}x^2\,d\nu(x)}{\sigma(\epsilon)^2}\cdot\frac{\epsilon}{\sigma(\epsilon)}=\int_0^{t_i}\bigl|h_{t_i}(s)\bigr|^3\,ds\cdot\frac{\epsilon}{\sigma(\epsilon)}\xrightarrow[\epsilon\to0]{}0,
\end{align*}
by hypothesis. Therefore, condition \eqref{2cond2} is verified.

\bigskip

Finally, using properties of the It$\hat{\text{o}}$ integral,
\begin{align*}
\Cov{F_{t_i}^{\epsilon}}{F_{t_j}^{\epsilon}}&=\mean{F_{t_i}^{\epsilon}F_{t_j}^{\epsilon}}=\frac{\dint_{[0,t_i]\times\{\abs{x}\leq\epsilon\}}h_{t_i}(s)h_{t_j}(s)x^2\,d\nu(x)\,ds}{\sigma(\epsilon)^2}=\int_0^{t_i\wedge t_j}h_{t_i}(s)h_{t_j}(s)\,ds\cdot\frac{\int_{\{\abs{x}\leq \epsilon\}}x^2\,d\nu(x)}{\sigma(\epsilon)^2}\\
&=\mean{\widehat{W}(h_{t_i})\widehat{W}(h_{t_j})}=\Cov{\widehat{W}(h_{t_i})}{\widehat{W}(h_{t_j})}.
\end{align*}
Hence, for all $\epsilon$, $\Var{F_\epsilon}=\Sigma:=\Var{Z}$, corroborating condition \eqref{2cond5} and concluding the proof.
\end{proof}

\bigskip

We have just shown that $X_t\stackrel{\text{law}}{\approx}\sigma(\epsilon)\widehat{W}(h_t)+N_t^{\epsilon}$, hence a process $X_t=I_1\bigl(h_t1_{[0,t]}\bigr)$ with infinitely many jumps can be substituted by a process $\sigma(\epsilon)\widehat{W}(h_t)+N_t^{\epsilon}$ with finitely many jumps, for $\epsilon$ small enough, making it easier to simulate.

\begin{preremark}
If $h_t(s,x)$ is also space dependent, then the previous result does not follow. Nevertheless, we can show that for a fixed $t$ the random variable $F_t^{\epsilon}\stackrel{law}{\longrightarrow}Z\sim\mathcal{N}\left(0,1\right)$ as $\epsilon\rightarrow0$ if
\begin{align}\label{smalljumpsconv0}
\frac{\dint_{[0,t]\times\{\abs{x}\leq\epsilon\}}\abs{xh_t(s,x)}^3\,d\nu(x)\,ds}{\widetilde{\sigma}_t(\epsilon)^3}\xrightarrow[\epsilon\rightarrow0]{}0,
\end{align}
where $\widetilde{\sigma}_t(\epsilon)^2=\Lqnorm{h_t1_{[0,t]\times\{\abs{x}\leq\epsilon\}}}{\mu}^2=\dint_{[0,t_i]\times\{\abs{x}\leq\epsilon\}}\bigl(h_t(s,x)\bigr)^2x^2\,d\nu(x)\,ds$. Conditions \eqref{2cond3} and \eqref{2cond4} are again trivially fulfilled. Also $\Var{F_t^{\epsilon}}=\Hnorm{\boldsymbol{D}F_t^{\epsilon}}^2=1$ which satisfies \eqref{2cond1} and \eqref{2cond5}. Hypothesis \eqref{smalljumpsconv0} implies condition \eqref{2cond2} is valid.
\end{preremark}

\subsubsection{Example: Fractional L\'evy Process}

The condition $\lim_{\epsilon\rightarrow0}\frac{\sigma(\epsilon)}{\epsilon}=\infty$ may quite easily be verified. In fact, the measure $d\nu(x)=\abs{x}^{-(2+\delta)}1_{\{-a\leq x\leq b\}}dx$ for $\delta\in(-1,1)$ and $a,b>0$ (no jumps bigger than $b$ or smaller than $-a$) is such that $\lim_{\epsilon\rightarrow0}\frac{\sigma(\epsilon)}{\epsilon}=\infty$. To check this, note that $\sigma(\epsilon)^2=\int_{\{\abs{x}\leq\epsilon\}}\abs{x}^2\,d\nu(x)=\frac{2\epsilon^{1-\delta}}{(1-\delta)}$, so $\lim_{\epsilon\rightarrow0}\frac{\sigma(\epsilon)^2}{\epsilon^2}=\lim_{\epsilon\rightarrow0}\frac{2}{(1-\delta)\epsilon^{1+\delta}}=\infty$. Hence, for the example, assume that the measure $\nu$ is such that $\lim_{\epsilon\rightarrow0}\frac{\sigma(\epsilon)}{\epsilon}=\infty$.

\begin{enumerate}

\item {\bf Fractional L\'evy Process (fLp):}\\
There are two ways to represent a fractional Brownian motion as an integral of a kernel with respect to Brownian motion (see \cite{Celine} for a thorough explanation), and both deliver the same process in the sense that both are Gaussian processes with the same covariance structure for $t\geq0$. One is the so-called Mandelbrot-Van Ness representation which is an integral over the whole real line with respect to a two sided Brownian motion. That is, if $B_t$ is a two-sided Brownian motion, then for all $t\in\RR$ and $H\in(0,1)$,
$$B_t^H=\int_{-\infty}^tC_H\bigl((t-s)_+^{H-\frac{1}{2}}-(-s)_+^{H-\frac{1}{2}}\bigr)\,dB_s,$$
where $C_H=\frac{\left(2H\sin(\pi H\Gamma(2H))\right)}{\Gamma(H+\frac{1}{2})}$ and $\Gamma$ is the Gamma function.\\
Alternatively, it can be represented in a compact interval by using the so-called Molchan-Golosov representation, that is, for all $t\geq0$,
$$B_t^H=\int_0^tK_H(t,s)\,dB_s,$$
where
\begin{align*}
K_H(t,s)=\begin{cases}
c^{(1)}_H\left[\lp\frac{t}{s}\rp^{H-\frac{1}{2}}(t-s)^{H-\frac{1}{2}}-\lp H-\frac{1}{2}\rp s^{\frac{1}{2}-H}\int_s^tu^{H-\frac{3}{2}}(u-s)^{H-\frac{1}{2}}\,du\right], &  H<\frac{1}{2}\\
c^{(2)}_Hs^{\frac{1}{2}-H}\int_s^t(u-s)^{H-\frac{3}{2}}u^{H-\frac{1}{2}}\,du, &  H>\frac{1}{2}
\end{cases},
\end{align*}
with $c^{(1)}_H=\sqrt{\frac{2H}{(1-2H)\beta(1-2H,H+1/2)}}$ and $c^{(2)}_H=\sqrt{\frac{H(2H-1)}{\beta(2-2H,H-1/2)}}$. Here $\beta$ denotes the beta function.

\bigskip

In the L\'evy case, following this same construction but substituting the Brownian motion with a \Levy\ process, in \cite{Heikki} the authors prove that these representations imply different processes with very different characteristics. Because of this ``non-uniqueness'', the fLp generated by the Mandelbrot-Van Ness representation is called fLpMvN, and the one generated by the Molchan-Golosov representation is called fLpMG.
\begin{defn}
Let $\Levyp_t$ be a two-sided \Levy\ process such that $\mean{\Levyp_1}=0$ and $\mean{\Levyp_1^2}<\infty$. For $H\in(0,1)$ and all $t\in\RR$, the process
$$\Levyp_t^H=\int_{-\infty}^tC_H\bigl((t-s)_+^{H-\frac{1}{2}}-(-s)_+^{H-\frac{1}{2}}\bigr)\,d\Levyp_s,$$
is the fractional \Levy\ process with the Mandelbrot-Van Ness transformation (fLpMvN). Furthermore, for $H\in(0,1)$ and all $t\geq0$, the process
$$\Levyp_t^H=\int_0^tK_H(t,s)\,d\Levyp_s,$$
is the fractional \Levy\ process with the Molchan-Golosov transformation (fLpMG).
\end{defn}

It is known that fLp's have the same covariance structure as fBm. The advantage of fLpMvN over fLpMG is that the former is stationary and the latter is not in general, as is shown in \cite{Heikki}. Nevertheless, since fLpMG is derived on a compact interval, Malliavin calculus can be applied to it.

Consider $\Levyp_t^H$ as an fLpMG, that is, $\Levyp_t^H=I_1(K^H_t)$ where
$$\Lprod{K^H_t}{K^H_s}{}=\frac{1}{2}\bigl(\abs{t}^{2H}+\abs{s}^{2H}-\abs{t-s}^{2H}\bigr).$$
According to Theorem \ref{smalljumpsconv}, since $\lim_{\epsilon\rightarrow0}\frac{\sigma(\epsilon)}{\epsilon}=\infty$ and $h_t(s,x)=K^H_t(s)$, it follows that $F_t^{\epsilon}\stackrel{law}{\approx}\widehat{W}(K^H_t)$ as $\epsilon \to 0$. But $\widehat{W}(K^H_t)=B^H_t$ is a fractional Brownian motion. We conclude that in order to simulate an fLpMG $X_t$, we just need to fix $\epsilon$ small enough, and simulate the finitely many jumps part $N_t^{\epsilon}=I_1(K^H_t1_{\{\abs{x}\geq\epsilon\}})$ along with an (independent) fBm part $B^H_t=\widehat{W}(K^H_t)$, because $\Levyp_t^H\stackrel{\text{law}}{\approx}\sigma(\epsilon)B^H_t+N_t^{\epsilon}$.

\end{enumerate}

\subsection{The Wiener-Poisson space case:\\ Product of O-U processes}

Finally, we use the second order Poincar\'e inequality developed in the combined space to obtain a central limit theorem for mixed processes.
\begin{defn}
Let $\Levyp_t$ be a \Levy\ process such that $\mean{\Levyp_1}=0$ and $\mean{\Levyp_1^2}<\infty$. For all $t\geq0$, the process
$$X_t=\sqrt{2\lambda}\int_{-\infty}^te^{-\lambda(t-u)}\,d\Levyp_u,$$
is the \Levy\ Ornstein-Uhlenbeck process.
\end{defn} 
In particular, for the characteristic triplet $(0,\sigma,\nu)$:
\begin{itemize}

\item If $\nu\equiv0$, $X_t$ is the classic Wiener Ornstein-Uhlenbeck process.

\item If $\sigma=0$, $X_t$ is the Poisson Ornstein-Uhlenbeck process.

\end{itemize}
First, notice that if we have a double Ornstein-Uhlenbeck (O-U) process as a sum of a Wiener O-U process $Y_t$ plus a Poisson O-U process $Z_t$ (independent of $Y_t$), one can prove in two different ways that the functional $F_T=T^{-\frac{1}{2}}\int_0^T\bigl(Y_t+Z_t\bigr)\,dt$ converges to a normal random variable as $T\to\infty$. The first way is to separate $F_T$ into the terms $T^{-\frac{1}{2}}\int_0^TY_t\,dt$ and $T^{-\frac{1}{2}}\int_0^TZ_t\,dt$ and use the NP Bounds in the Wiener space and in the Poisson space respectively to prove that each part goes to a normal. The other method is to use inequality (\ref{mainineq1}) (NP Bound in the Wiener-Poisson space) and just do one computation. The second way is clearly faster (since the kernels are the same). This is one advantage of having the inequality in the combined space.

\bigskip

For our example, let us focus on a process that cannot be tackled by either of the NP Bounds (in the Wiener or Poisson spaces) to prove a CLT. First, for simplicity's sake (to avoid dealing with constants), assume that the triplet for the underlying L\'evy process is given by $(0,1,\nu)$, where $\int_\RR x^2\,d\nu(x)=1$. Moreover, assume that $\int_{\RR_0}x^4\,d\nu(x)<\infty$. Let $Y_t=\int_0^t\sqrt{2\lambda}e^{-\lambda(t-s)}\,dW_s$ and  $Z_t=\dint_{[0,t]\times\RR_0}\sqrt{2\lambda}e^{-\lambda(t-s)}x\,d\widetilde{N}(s,x)$, so $Y_t$ is a Wiener O-U process and $Z_t$ is a Poisson (pure jump) O-U process. If $h_t(s)=\sqrt{2\lambda}e^{-\lambda(t-s)}1_{\{s\leq t\}}$ then the double O-U process mentioned above is just $Y_t+Z_t=I_1(h_t)$. Now, define $h_t^{(0)}(s,x)=h_t(s)1_{\{x=0\}}(x)$ and $h_t^{(1)}(s,x)=h_t(s)1_{\{x\neq0\}}(x)$, then $Y_t=I_1\bigl(h_t^{(0)}\bigr)$ and $Z_t=I_1\bigl(h_t^{(1)}\bigr)$. Notice that due to the normalization of the L\'evy triplet we have that $C(t,s)=\Lqprod{h_t^{(0)}}{h_s^{(0)}}{\mu}=\Lqprod{h_t^{(1)}}{h_s^{(1)}}{\mu}$. The goal of this subsection is to show that $F_T=T^{-\frac{1}{2}}\int_0^T\bigl(Y_tZ_t\bigr)\,dt\stackrel{law}{\longrightarrow}Z\sim\mathcal{N}(0,\Sigma^2)$ as $T\to\infty$.

\bigskip

Since $h_t^{(0)}$ and $h_t^{(1)}$ have disjoint supports $\bigl($and using the product formula (\ref{prodint})$\bigr)$, $Y_tZ_t=I_2\bigl(h_t^{(0)}\widetilde{\otimes}h_t^{(1)}\bigr)$, and by Fubini $F_T=I_2\bigl(T^{-\frac{1}{2}}\int_{\cdot\vee\cdot}^Th_t^{(0)}\widetilde{\otimes}h_t^{(1)}\,dt\bigr)$. Hence $F_T$ lies in the $2^{\text{nd}}$ chaos. According to Corollary \ref{mainqchaos} we just need to check conditions (\ref{2cond1}), (\ref{2cond2}), (\ref{2cond3}), (\ref{2cond4}) and (\ref{2cond5}).
\begin{itemize}

\item {\bf Expectation of the First Derivative Norm:}\\
\begin{itemize}

\item First Malliavin Derivative:
$$\boldsymbol{D}_zF_T=T^{-\frac{1}{2}}\int_0^TI_1\bigl(h_t^{(0)}\bigr)h_t^{(1)}(z)+I_1\bigl(h_t^{(1)}\bigr)h_t^{(0)}(z)\,dt.$$

\item Norm of the First Malliavin Derivative:
\begin{align*}
\Hnorm{\boldsymbol{D}_zF_T}^2&=T^{-1}\int_{[0,T]^2}\Hprod{I_1\bigl(h_t^{(0)}\bigr)h_t^{(1)}(z)+I_1\bigl(h_t^{(1)}\bigr)h_t^{(0)}(z)}{I_1\bigl(h_s^{(0)}\bigr)h_s^{(1)}(z)+I_1\bigl(h_s^{(1)}\bigr)h_s^{(0)}(z)}\,dt\,ds\\
&=T^{-1}\int_{[0,T]^2}I_1\bigl(h_t^{(0)}\bigr)I_1\bigl(h_s^{(0)}\bigr)\Hprod{h_t^{(1)}}{h_s^{(1)}}+I_1\bigl(h_t^{(1)}\bigr)I_1\bigl(h_s^{(1)}\bigr)\Hprod{h_t^{(0)}}{h_s^{(0)}}\,dt\,ds\\
&=T^{-1}\int_{[0,T]^2}\bigl(I_1\bigl(h_t^{(0)}\bigr)I_1\bigl(h_s^{(0)}\bigr)+I_1\bigl(h_t^{(1)}\bigr)I_1\bigl(h_s^{(1)}\bigr)\bigr)C(t,s)\,dt\,ds,
\end{align*}
then,
\begin{align*}
\Hnorm{\boldsymbol{D}_zF_T}^4&\leq 2T^{-2}\biggl[\biggl(\int_{[0,T]^2}I_1\bigl(h_t^{(0)}\bigr)I_1\bigl(h_s^{(0)}\bigr)C(t,s)\,dt\,ds\biggr)^2+\biggl(\int_{[0,T]^2}I_1\bigl(h_t^{(1)}\bigr)I_1\bigl(h_s^{(1)}\bigr)C(t,s)\,dt\,ds\biggr)^2\biggr]\\
&=2T^{-2}\int_{[0,T]^4}\biggl[\prod_{i=1}^4I_1\bigl(h_{t_i}^{(0)}\bigr)+\prod_{i=1}^4I_1\bigl(h_{t_i}^{(1)}\bigr)\biggr]C(t_1,t_2)C(t_3,t_4)\,dt_1\,dt_2\,dt_3\,dt_4.
\end{align*}

\item Expectation of the First Derivative Norm:\\
Notice that by the product formula (\ref{prodint}) we have that
$$\prod_{i=1}^4I_1\bigl(h_{t_i}^{(0)}\bigr)=\biggl[\Hprod{h_{t_1}^{(0)}}{h_{t_2}^{(0)}}+I_2\bigl(h_{t_1}^{(0)}\widetilde{\otimes}h_{t_2}^{(0)}\bigr)\biggr]\biggl[\Hprod{h_{t_3}^{(0)}}{h_{t_4}^{(0)}}+I_2\bigl(h_{t_3}^{(0)}\widetilde{\otimes}h_{t_4}^{(0)}\bigr)\biggr],$$
and
$$\prod_{i=1}^4I_1\bigl(h_{t_i}^{(1)}\bigr)=\biggl[\Hprod{h_{t_1}^{(1)}}{h_{t_2}^{(1)}}+I_1\bigl(h_{t_1}^{(1)}\otimes_0^1h_{t_2}^{(1)}\bigr)+I_2\bigl(h_{t_1}^{(1)}\widetilde{\otimes}h_{t_2}^{(1)}\bigr)\biggr]\biggl[\Hprod{h_{t_3}^{(1)}}{h_{t_4}^{(1)}}+I_1\bigl(h_{t_3}^{(1)}\otimes_0^1h_{t_4}^{(1)}\bigr)+I_2\bigl(h_{t_3}^{(1)}\widetilde{\otimes}h_{t_4}^{(1)}\bigr)\biggr],$$
so
$$\mean{\prod_{i=1}^4I_1\bigl(h_{t_i}^{(0)}\bigr)}=C(t_1,t_2)C(t_3,t_4)+\HHprod{h_{t_1}^{(0)}\widetilde{\otimes}h_{t_2}^{(0)}}{h_{t_3}^{(0)}\widetilde{\otimes}h_{t_4}^{(0)}},$$
and
$$\mean{\prod_{i=1}^4I_1\bigl(h_{t_i}^{(1)}\bigr)}=C(t_1,t_2)C(t_3,t_4)+\Hprod{h_{t_1}^{(1)}\otimes_0^1h_{t_2}^{(1)}}{h_{t_3}^{(1)}\otimes_0^1h_{t_4}^{(1)}}+\HHprod{h_{t_1}^{(1)}\widetilde{\otimes}h_{t_2}^{(1)}}{h_{t_3}^{(1)}\widetilde{\otimes}h_{t_4}^{(1)}}.$$
Also notice that
$$C(t,s)=\int_0^{t\wedge s}2\lambda e^{-\lambda(t+s-2u)}\,du=e^{-\lambda\abs{t-s}}-e^{-\lambda(t+s)}\leq e^{-\lambda\abs{t-s}}\leq1,$$
$$\underbrace{\Hprod{h_{t_1}^{(1)}\otimes_0^1h_{t_2}^{(1)}}{h_{t_3}^{(1)}\otimes_0^1h_{t_4}^{(1)}}}_{\Hprod{xh_{t_1}^{(1)}h_{t_2}^{(1)}}{xh_{t_3}^{(1)}h_{t_4}^{(1)}}}=\int_0^{\min\{t_1,t_2,t_3,t_4\}}\int_{\RR_0}x^44\lambda^2e^{-\lambda(t_1+t_2+t_3+t_4-4u)}\,d\nu(x)\,du\leq\lambda\int_{\RR_0}x^4\,d\nu(x),$$
$$\HHprod{h_{t_1}^{(0)}\widetilde{\otimes}h_{t_2}^{(0)}}{h_{t_3}^{(0)}\widetilde{\otimes}h_{t_4}^{(0)}}=\HHprod{h_{t_1}^{(1)}\widetilde{\otimes}h_{t_2}^{(1)}}{h_{t_3}^{(1)}\widetilde{\otimes}h_{t_4}^{(1)}}=\frac{C(t_1,t_3)C(t_2,t_4)+C(t_1,t_4)C(t_2,t_3)}{2}\leq1.$$
Putting all this together we get
\begin{align*}
\mean{\Lqnorm{\boldsymbol{D}F_T}{\mu}^4}&=2T^{-2}\int_{[0,T]^4}\biggl(\mean{\prod_{i=1}^4I_1\bigl(h_{t_i}^{(0)}\bigr)}+\mean{\prod_{i=1}^4I_1\bigl(h_{t_i}^{(1)}\bigr)}\biggr)C(t_1,t_2)C(t_3,t_4)\,dt_1\,dt_2\,dt_3\,dt_4\\
&\leq2\biggl(4+\lambda\int_{\RR_0}x^4\,d\nu(x)\biggr)\biggl(T^{-1}\int_{[0,T]^2}C(t,s)\,dt\,ds\biggr)^2\\
&\leq2\biggl(4+\lambda\int_{\RR_0}x^4\,d\nu(x)\biggr)\biggl(T^{-1}\int_{[0,T]^2}e^{-\lambda\abs{t-s}}\,dt\,ds\biggr)^2\\
&=2\biggl(4+\lambda\int_{\RR_0}x^4\,d\nu(x)\biggr)\biggl(2T^{-1}\int_0^T\int_0^te^{-\lambda(t-s)}\,ds\,dt\biggr)^2\\
&\leq2\biggl(4+\lambda\int_{\RR_0}x^4\,d\nu(x)\biggr)\biggl(\frac{2}{\lambda}\biggr)^2.
\end{align*}
\end{itemize}
All this proves that
\begin{align*}
\boxed{\mean{\Hnorm{\boldsymbol{D}F_T}^4}^{\frac{1}{4}}=O(1)\ \text{ as }\ T\to\infty.}
\end{align*}

\item {\bf Expectation of the Cube of the First Derivative Norm:}
\begin{itemize}

\item Cube of the First Malliavin Derivative:\\
Since $h_{t}^{(0)}(z)\cdot h_{s}^{(1)}(z)=0$ for all $z\in\RR^+\times\RR$ then,
$$\abs{\boldsymbol{D}F_T}^3=\abs{T^{-\frac{3}{2}}\int_{[0,T]^3}\biggl[\prod_{i=1}^3I_1\bigl(h_{t_i}^{(0)}\bigr)h_{t_i}^{(1)}+\prod_{i=1}^3I_1\bigl(h_{t_i}^{(1)}\bigr)h_{t_i}^{(0)}\biggr]\,dt_1\,dt_2\,dt_3}.$$

\item Norm of the Cube of the First Malliavin Derivative:\\
Since $x\cdot h_{t}^{(0)}(z)=0$ for all $z=(t,x)\in\RR^+\times\RR$ then
$$\Hprod{\abs{x}}{\abs{\boldsymbol{D}F_T}^3}\leq T^{-\frac{3}{2}}\int_{[0,T]^3}\prod_{i=1}^3\abs{I_1\bigl(h_{t_i}^{(0)}\bigr)}\Hprod{\abs{x}}{\prod_{i=1}^3\abs{h_{t_i}^{(1)}}}\,dt_1\,dt_2\,dt_3.$$

\item Expectation of the Cube of the First Derivative Norm:\\
Notice that by H$\ddot{\text{o}}$lder's inequality we have
\begin{align*}
\mean{\prod_{i=1}^3\abs{I_1\bigl(h_{t_i}^{(0)}\bigr)}}&\leq \mathbb{E}\bigl[\bigl(\hspace{-.7cm}\overbrace{I_1\bigl(h_{t_1}^{(0)}\bigr)I_1\bigl(h_{t_2}^{(0)}\bigr)}^{\Hprod{h_{t_1}^{(0)}}{h_{t_2}^{(0)}}+I_2\bigl(h_{t_1}^{(0)}\widetilde{\otimes}h_{t_2}^{(0)}\bigr)}\hspace{-.7cm}\bigr)^2\bigr]^{\frac{1}{2}}\mean{\bigl(I_1\bigl(h_{t_3}^{(0)}\bigr)\bigr)^2}^{\frac{1}{2}}\\
&\leq \biggl(\abs{C(t_1,t_2)}+\mean{\bigl(I_2\bigl(h_{t_1}^{(0)}\widetilde{\otimes}h_{t_2}^{(0)}\bigr)\bigr)^2}^{\frac{1}{2}}\biggr)\mean{\bigl(I_1\bigl(h_{t_3}^{(0)}\bigr)\bigr)^2}^{\frac{1}{2}}\\
&\leq\biggl(\abs{C(t_1,t_2)}+\HHnorm{h_{t_1}^{(0)}\otimes h_{t_2}^{(0)}}\biggr)\Hnorm{h_{t_3}^{(0)}}\leq2,\\
\end{align*}
and
\begin{align*}
\Hprod{\abs{x}}{\prod_{i=1}^3\abs{h_{t_i}^{(1)}}}&=\int_0^{\min\{t_1,t_2,t_3\}}(2\lambda)^{\frac{3}{2}}e^{-\lambda(t_1+t_2+t_3-3u)}\,du\int_{\RR_0}\abs{x}^3\,d\nu(x)\\
&\leq\frac{2\sqrt{2\lambda}}{3}\biggl(\int_{\RR_0}\abs{x}^3\,d\nu(x)\biggr)e^{-\lambda(t_1+t_2+t_3-3\min\{t_1,t_2,t_3\})}.
\end{align*}
Putting all this together we get
\begin{align*}
\mean{\Hprod{\abs{x}}{\abs{\boldsymbol{D}F_T}^3}}&\leq T^{-\frac{3}{2}}\int_{[0,T]^3}\mean{\prod_{i=1}^3\abs{I_1\bigl(h_{t_i}^{(0)}\bigr)}}\Hprod{\abs{x}}{\prod_{i=1}^3\abs{h_{t_i}^{(1)}}}\,dt_1\,dt_2\,dt_3\\
&\leq \frac{4\sqrt{2\lambda}}{3}\biggl(\int_{\RR_0}\abs{x}^3d\nu(x)\biggr)T^{-\frac{3}{2}}\int_{[0,T]^3}e^{-\lambda(t_1+t_+t_3-3\min\{t_1,t_2,t_3\})}\,dt_1\,dt_2\,dt_3\\
&=\frac{24\sqrt{2\lambda}}{3}\biggl(\int_{\RR_0}\abs{x}^3d\nu(x)\biggr)T^{-\frac{3}{2}}\int_0^T\int_0^t\int_0^se^{-\lambda(t+s-2u)}\,du\,ds\,dt\\
&\leq \frac{4\sqrt{2}\bigl(\int_{\RR_0}\abs{x}^3d\nu(x)\bigr)}{\lambda^{\frac{3}{2}}\sqrt{T}}=O(T^{-\frac{1}{2}}).
\end{align*}

\end{itemize}
All this proves that
\begin{align*}
\boxed{\mean{\Hprod{\abs{x}}{\abs{\boldsymbol{D}F}^3}}\rightarrow0\ \text{ as } \ T\rightarrow\infty.}
\end{align*}

\item {\bf Expectation of the Contraction Norm:}
\begin{itemize}

\item Second Malliavin Derivative:
$$\boldsymbol{D}^2_{z_1,z_2}F_T=T^{-\frac{1}{2}}\int_0^Th_{t}^{(1)}(z_1)h_{t}^{(0)}(z_2)+h_{t}^{(0)}(z_1)h_{t}^{(1)}(z_2)\,dt.$$

\item Contraction of order 1:
\begin{align*}
\boldsymbol{D}^2F_T\otimes_1\boldsymbol{D}^2F_T&=T^{-1}\int_{[0,T]}h_{t}^{(1)}(z_1)h_{s}^{(1)}(z_1)\overbrace{\Hprod{h_{t}^{(0)}}{h_{s}^{(0)}}}^{=\ C(t,s)\leq1}+h_{t}^{(0)}(z_1)h_{s}^{(0)}(z_1)\overbrace{\Hprod{h_{t}^{(1)}}{h_{s}^{(1)}}}^{=\ C(t,s)\leq1}\,dt\,ds\\
&\leq T^{-1}\int_{[0,T]}h_{t}^{(1)}(z_1)h_{s}^{(1)}(z_1)+h_{t}^{(0)}(z_1)h_{s}^{(0)}(z_1)\,dt\,ds.
\end{align*}

\item Norm of the Contraction:
$$\Hnorm{\boldsymbol{D}^2F_T\otimes_1\boldsymbol{D}^2F_T}^2\leq T^{-2}\int_{[0,T]^4}\Hprod{h_{t_1}^{(0)}h_{t_2}^{(0)}}{h_{t_3}^{(0)}h_{t_4}^{(0)}}+\Hprod{h_{t_1}^{(1)}h_{t_2}^{(1)}}{h_{t_3}^{(1)}h_{t_4}^{(1)}}\,dt_1\,dt_2\,dt_3\,dt_4.$$

\item Expectation of the Contraction Norm:\\
Notice that
$$\Hprod{h_{t_1}^{(0)}h_{t_2}^{(0)}}{h_{t_3}^{(0)}h_{t_4}^{(0)}}=\Hprod{h_{t_1}^{(1)}h_{t_2}^{(1)}}{h_{t_3}^{(1)}h_{t_4}^{(1)}}=\lambda \bigl(e^{-\lambda(t_1+t_2+t_3+t_4-4\min\{t_1,t_2,t_3,t_4\})}-e^{-\lambda(t_1+t_2+t_3+t_4)}\bigr),$$
so
$$\int_{[0,T]^4}\Hprod{h_{t_1}^{(i)}h_{t_2}^{(i)}}{h_{t_3}^{(i)}h_{t_4}^{(i)}}\,dt_1\,dt_2\,dt_3\,dt_4\leq24\lambda\int_0^T\int_0^t\int_0^s\int_0^ue^{-\lambda(t+s+u-3v)}\,dv\,du\,ds\,dt\leq \frac{4T}{\lambda^2}.$$
Putting all this together we get
$$\mean{\Hnorm{\boldsymbol{D}^2F_T\otimes_1\boldsymbol{D}^2F_T}^2}\leq 2T^{-2}\int_{[0,T]^4}\Hprod{h_{t_1}^{(i)}h_{t_2}^{(i)}}{h_{t_3}^{(i)}h_{t_4}^{(i)}}\,dt_1\,dt_2\,dt_3\,dt_4\leq T^{-2}\frac{8T}{\lambda^2}=\frac{8}{T\lambda^2}.$$

\end{itemize}
All this proves that
\begin{align*}
\boxed{\mean{\Hnorm{\boldsymbol{D}^2F_T\otimes_1\boldsymbol{D}^2F_T}^2}\rightarrow0\ \text{ as } \ T\rightarrow\infty.}
\end{align*}

\item {\bf Expectation of the Squared Second Derivative Norm:}
\begin{itemize}

\item Square of the Second Malliavin Derivative:
$$(\boldsymbol{D}^2_{z_1,z_2}F_T)^2=T^{-1}\int_{[0,T]^2}h_t^{(0)}(z_1)h_s^{(0)}(z_1)h_t^{(1)}(z_2)h_s^{(1)}(z_2)+h_t^{(1)}(z_1)h_s^{(1)}(z_1)h_t^{(0)}(z_2)h_s^{(0)}(z_2)\,dt\,ds.$$

\item Inner Product of the Squared Second Malliavin Derivative:\\
Like in a previous calculation, $x\cdot h_t^{(0)}=0$, so
$$\Hprod{x}{(\boldsymbol{D}^2F_T)^2}=T^{-1}\int_{[0,T]^2}h_t^{(0)}h_s^{(0)}\overbrace{\Hprod{x}{h_t^{(1)}h_s^{(1)}}}^{\int_{\RR_0}x^3d\nu(x)C(t,s)}\,dt\,ds\leq \int_{\RR_0}\abs{x}^3d\nu(x)T^{-1}\int_{[0,T]^2}h_t^{(0)}h_s^{(0)}\,dt\,ds.$$

\item Expectation of the Squared Second Derivative Norm:\\
By the earlier computations we have,
$$\mean{\Hnorm{\Hprod{x}{(\boldsymbol{D}^2F_T)^2}}^2}\leq\biggl(\int_{\RR_0}\abs{x}^3\,d\nu(x)\biggr)^2T^{-2}\int_{[0,T]^4}\Hprod{h_{t_1}^{(0)}h_{t_2}^{(0)}}{h_{t_3}^{(0)}h_{t_4}^{(0)}}\,dt_1\,dt_2\,dt_3\,dt_4\leq\frac{4}{\lambda^2T}.$$

\end{itemize}
All this proves that
\begin{align*}
\boxed{\mean{\Hnorm{\Hprod{x}{(\boldsymbol{D}^2F)^2}}^2}\rightarrow0\ \text{ as } \ T\rightarrow\infty.}
\end{align*}

\item {\bf Existence of the Variance:}\\

Since $F_T=I_2\bigl(T^{-\frac{1}{2}}\int_{\cdot\vee\cdot}^Th_t^{(0)}\widetilde{\otimes}h_t^{(1)}dt\bigr)$ then
\begin{align*}
\text{Var}[F_T]&=\HHnorm{T^{-\frac{1}{2}}\int_{\cdot\vee\cdot}^Th_t^{(0)}\widetilde{\otimes}h_t^{(1)}dt}^2=T^{-1}\int_{[0,T]^2}\biggl(\int_{s_1\vee s_2}^T2\lambda e^{-\lambda(2t-s_1-s_2)}dt\biggr)^2\,ds_1\,ds_2\\
&= T^{-1}\int_{[0,T]^2}\bigl(e^{-\lambda\abs{s_1-s_2}}-e^{-\lambda(2T-s_1-s_2)}\bigr)^2\,ds_1\,ds_2\\
&= T^{-1}\int_{[0,T]^2}e^{-2\lambda\abs{s_1-s_2}}-2e^{-2\lambda(T-s_1\wedge s_2)}+e^{-2\lambda(2T-s_1-s_2)}\,ds_1\,ds_2\\
&= T^{-1}\biggl(\frac{T}{\lambda}-6\frac{(1-e^{-2\lambda T})}{4\lambda^2}+4Te^{-2\lambda T}+\frac{(1-e^{-2\lambda T})^2}{4\lambda^2}\biggr)=\frac{1}{\lambda}+O(T^{-1}).
\end{align*}
All this proves that,
\begin{align*}
\boxed{\text{Var}[F_T]\rightarrow \frac{1}{\lambda}\in(0,\infty)\ \text{ exists as }\ T\to\infty.}
\end{align*}

\end{itemize}
Since all five conditions are met, by Corollary \ref{mainqchaos} we have that $F_T\stackrel{law}{\longrightarrow}Z\sim\mathcal{N}(0,\frac{1}{\lambda})$ as $T\to\infty$. Moreover, due to the quantitative property of the inequality, one can estimate (from the computations above) that the rate of convergence to normality is at least $O\bigl(T^{-\frac{1}{4}}\bigr)$, i.e., $d_\PZclass\bigl(\frac{F_T}{\sqrt{\text{Var}[F_T]}},Z\bigr)=O\bigl(T^{-\frac{1}{4}}\bigr)$ as $T\to\infty$. This rate is similar to the one obtained for the linear functionals of Gaussian-subordinated fields with an underlying process given by the increments of fBm or the fractional-driven O-U, when $H\in\bigl(0,\frac{1}{2}\bigr)$.

\end{document}